\documentclass[11pt]{amsart}
\usepackage{amsmath}
\usepackage{amssymb}
\usepackage{amsthm}
\usepackage{amscd}
\usepackage{fontenc}
\usepackage{url}
\usepackage[all]{xy}

\def\acts{\curvearrowright}
\numberwithin{equation}{section}

\newcounter{AbcT}

\newtheorem {Theoremintro}    {Theorem}

\newtheorem {Theorem}    {Theorem}[section]
\newtheorem* {Theorem2}    {Theorem 2}

\newtheorem* {Theorem5}    {Theorem 1(a)}

\newtheorem* {Theorem7}    {Theorem 1(b)}
\newtheorem* {Theorem8}    {Theorem 1(c)}
\newtheorem* {Theorem9}    {Theorem 1(d)}
\newtheorem* {Theorem10}    {Theorem 3}
\def\acts{\curvearrowright}

\newtheorem {Problem}    [Theorem]{Problem}
\newtheorem {Question}    [Theorem]{Question}

\newtheorem {Lemma}      [Theorem]    {Lemma}
\newtheorem {Corollary}   [Theorem] {Corollary}
\newtheorem {Proposition}[Theorem]    {Proposition}

\theoremstyle{remark}
\newtheorem {Remark}		 [Theorem]    {\bf{Remark}}
\newtheorem {Remarks}		 [Theorem]    {\bf{Remarks}}
\newtheorem {Definition} [Theorem]    {\bf{Definition}}

\newcommand{\Amen}{\mathrm{Amen}}
\newcommand{\Fin}{\mathrm{Fin}}
\newcommand{\Cay}{\mathit{Cay}}
\newcommand{\Cayu}{\mathit{Cay_{uo}}}
\newcommand{\defi}{\mathit{def}}

\newcounter{DM@bibnum}

\newcommand {\abtf}{{\rm ab,tf}}

\newcommand{\cost}{\mathcal{C}}

\newcommand {\Z} {{\mathbb Z}}

\newcommand{\la}{\langle}
\newcommand{\ra}{\rangle}

\def\deg{{\rm deg\,}}

\def\NSL_2{{\mathcal N SL_2}}

\def\eps{\varepsilon}

\def\phi{\varphi}

\def\calC{{\mathcal C}}

\def\calI{{\mathcal I}}
\def\calJ{{\mathcal J}}

\def\calT{{\mathcal T}}

\def\hbar{\bar h}

\def\dbN{{\mathbb N}}

\def\dbR{{\mathbb R}}

\def\dbZ{{\mathbb Z}}

\newcommand{\sub}{\subseteq}
\newcommand{\iv}{^{-1}}
\newcommand{\lla}{\la\!\la}
\newcommand{\rra}{\ra\!\ra}
\def\skv{{\vskip .12cm}}

\begin{document}

\title{The Tarski numbers of groups}
\author{Mikhail Ershov}
\address{University of Virginia}
\email{ershov@virginia.edu}
\author{Gili Golan}
\address{Bar-Ilan University}
\email{gili.golan@math.biu.ac.il}

\author{Mark Sapir}
\address{Vanderbilt University}
\email{m.sapir@vanderbilt.edu}
\thanks{The work of the first author was partially supported by the NSF grant DMS-1201452
and the Sloan Research Fellowship grant BR 2011-105. The work of the second author was partially supported by BSF grant T-2012-238 and the Ministry of Science and Technology of Israel. The work of the third author was partially supported by NSF DMS-1318716, DMS-1261557 and BSF 2010295.}

\subjclass[2000]{Primary 43A07, 20F65
Secondary 20E18, 20F05, 20F50}
\keywords{Tarski number, paradoxical decomposition, amenability, $L^2$-Betti number, cost, Golod-Shafarevich groups}

\begin{abstract}
The Tarski number of a non-amenable group $G$ is the minimal number of pieces
in a paradoxical decomposition of $G$. In this paper we investigate how
Tarski numbers may change under various group-theoretic operations. Using these
 estimates and known properties of Golod-Shafarevich groups, we show that
the Tarski numbers of $2$-generated non-amenable groups can be arbitrarily large.
We also use the cost of group actions
 to show that there exist groups with Tarski numbers $5$ and 6.
These provide the first examples of non-amenable groups without free subgroups
whose Tarski number has been computed precisely.
\end{abstract}
\maketitle

\section{Introduction}

Recall the definition of a \emph{paradoxical decomposition} of a group.

\begin{Definition}\rm
\label{d1}
A group $G$ admits a \emph{paradoxical decomposition}
if there exist positive integers $m$ and $n$, disjoint subsets $P_1,\ldots, P_m,Q_1,\ldots, Q_n$ of $G$
and elements $g_1,\ldots, g_m$, $h_1,\ldots, h_n$ of $G$ such that
\begin{equation}\label{eq1}G=\bigcup_{i=1}^m P_ig_i=\bigcup_{j=1}^n Q_jh_j.\end{equation}
\end{Definition}
It is well known \cite{Wa} that $G$ admits a paradoxical decomposition if and only if it is non-amenable.
The minimal possible value of $m+n$ in a paradoxical decomposition of $G$
is called the \emph{Tarski number} of $G$ and denoted by $\mathcal T(G)$.

The definition stated above (with the elements $g_1,\dots,g_m,h_1,\dots,h_n$ acting on the left) appears both in \cite{Wa} and \cite{Sa}. A slightly different definition of a paradoxical decomposition
(see, for example, \cite{CSGH}) requires the sets $P_1,\ldots, P_m,$ $Q_1,\ldots, Q_n$ to cover the entire group $G$ and each of the unions $\bigcup_{i=1}^m P_ig_i$ and $\bigcup_{j=1}^n Q_jh_j$ to be disjoint. This alternative definition leads to the same notion of Tarski number:
this follows from the proof of \cite[Proposition~1.2]{RY} and Remark \ref{1} below, but for completeness we
will prove the equivalence of the two definitions of Tarski numbers in Appendix~A.

It is clear that for any paradoxical decomposition we must have $m\geq 2$ and $n\geq 2$,
so the minimal possible value of Tarski number is $4$. By a theorem of J\'onsson and Dekker (see, for example, \cite[Theorem~5.8.38]{Sa}),  $\calT(G)=4$
if and only if $G$ contains a non-Abelian free subgroup.

The problem of describing the set of Tarski numbers of groups has been formulated in \cite{CSGH}, and the following results have been proved there:
\newpage
\begin{Theorem}\label{t:csgh} $\empty$
\begin{itemize} \item[(i)] The Tarski number of any torsion group is at least $6$.
\item[(ii)] The Tarski number of any non-cyclic free Burnside group of odd exponent $\ge 665$ is between $6$ and $14$.
\end{itemize}
\end{Theorem}

For quite some time it was unknown if the set of Tarski numbers is infinite. That question was asked by Ozawa \cite{OS} and answered in the positive by the third author. For every $m\ge 1$ let $\Amen_m$ (resp. $\Fin_m$) be the class of all groups where all $m$-generated subgroups are amenable (resp. finite). For example, $\Amen_1$ is the class of all groups and $\Fin_1$ is the class of all torsion groups. Clearly $\Fin_m\subseteq \Amen_m$ for every $m$. Ozawa noticed \cite{OS} that all groups in $\Amen_m$ have Tarski number at least $m+3$, and the third author observed that $\Fin_m$ (for every $m$) contains non-amenable groups. This immediately follows from two results about Golod-Shafarevich groups proved in \cite{EJ} and \cite{EJ2} (see \S~\ref{s:l} below). Thus there exist non-amenable groups with arbitrarily large Tarski numbers.

In fact, results of \cite{EJ,EJ2} imply the following much stronger statement (see \S~\ref{sec:unbounded} for details).

\begin{Theorem}
\label{t:ej} There exists a finitely generated non-amenable group $H$ such that for every $m\ge 1$, $H$ has a finite index subgroup $H_m$ lying in $\Fin_m$ and hence $H_m$ is a non-amenable group with Tarski number at least $m+3$. Moreover, for every prime $p$ we can assume that $H$ is a residually-$p$ group.
\end{Theorem}

\begin{Remark} Since subgroups of finite index are quasi-isometric to the whole group,  Theorem \ref{t:ej} implies that for some natural number $t>4$, the property of having Tarski number $t$ is not invariant under quasi-isometry. We do not know what the number $t$ is, i.e., what is the Tarski number of the group $G$ from Theorem ~\ref{t:ej}. The only estimates we have are based on a rough calculation of the isoperimetric constant of the group $G$ which gives an enormous upper bound for $t$ (about $10^{10^8}$). Note that a well-known question of Benson Farb asks whether the property of finitely generated groups of having a non-Abelian free subgroup is invariant under quasi-isometry. In view of the result of J\'onsson and Dekker this is equivalent to the question whether the property of having Tarski number $4$ is invariant under quasi-isometry.
\end{Remark}

We now turn to the discussion of our results.

\subsection{Tarski numbers of subgroups and  quotients}
If $H$ is a non-amenable group which is either a subgroup or a quotient of a group $G$, it is easy
to see that $\calT(G)\leq \calT(H)$ (for a proof see \cite[Theorems~5.8.13, 5.8.16]{Sa}).
Conversely, in many cases it is possible to find an explicit upper bound on $\calT(H)$ in terms of $\calT(G)$.
Our results of this type are collected in the following theorem:
\begin{Theoremintro}
\label{Tar_combined}
Let $G$ be a non-amenable group and $H$ a subgroup of $G$.
\begin{itemize}
\item[(a)] Suppose that $H$ has finite index in $G$. Then
$$\calT(H)-2\leq [G:H](\calT(G)-2).$$
\item[(b)]
Let $\mathcal{V}$ be a variety of groups where all groups are amenable and relatively free groups are right orderable. Then there exists a function $f\colon \mathbb{N}\to \mathbb{N}$ (depending only on $\mathcal{V}$)
with the following property: if $H$ is normal in $G$ and $G/H\in \mathcal{V}$, then
$\calT(H)\leq f(\calT(G))$.
\item[(c)] Assume that $H$ is normal and amenable. Then $\calT(G/H)=\calT(G).$
\item[(d)]
Assume that $G=H\times K$ for some $K$. Then $\min\{\calT(H),\calT(K)\}\leq  2(\calT(G)-1)^2$.
\end{itemize}
\end{Theoremintro}

\begin{Remarks} \label{r:1} \rm
(i) There is an obvious similarity
between Theorem~\ref{Tar_combined}(a) and the Schreier index formula.
The main difference with the latter is that we do not know whether the above inequality
can become an equality (for $H\neq G$), and if not, how large the ratio
$\frac{\calT(H)-2}{\calT(G)-2}$ can be compared to the index $[G:H]$. We do know that
dependence on the index cannot be eliminated in view of Theorem ~\ref{t:ej}.

(ii) Varieties $\mathcal V$ for which the hypotheses of Theorem~\ref{Tar_combined}(b) hold include the variety of all Abelian groups and more generally all solvable groups of a given class \cite{Ord}.
In particular, if $G/H$ is cyclic, Theorem~\ref{Tar_combined}(b) yields a (non-trivial) lower
bound on $\calT(G)$ in terms of $\calT(H)$ alone (independent of the size of $G/H$).
This special case will be used to prove Theorem~\ref{unbounded_2gen} below.

(iii) We do not know if $\calT(H\times K)$ can be strictly smaller than the minimum of  $\calT(H)$ and $\calT(K)$.
By Theorem \ref{Tar_combined}(c) the inequality becomes an equality if one of the groups $H$ or $K$ is amenable.
For the case $K=H$ see \cite[Problem~5.9.23]{Sa}.
\end{Remarks}
\skv

\subsection{Unbounded Tarski numbers}
It is clear that non-amenable groups from $\Amen_m$ must have at least $m+1$ generators. Thus the already mentioned results  about groups  with arbitrarily large Tarski numbers give rise to the following natural question:

\begin{Question}
Is there a relation between the minimal number of generators of a non-amenable group and its Tarski number?
\end{Question}

The next theorem shows that the answer is negative.

\begin{Theoremintro}\label{unbounded_2gen}
The set of Tarski numbers of $2$-generated non-amenable groups is infinite.
Moreover, the set of Tarski numbers of $2$-generated infinite groups with property~(T) is infinite.
\end{Theoremintro}

To prove Theorem~\ref{unbounded_2gen} we use a construction from \cite{HH} to embed any finitely generated group $G$ from $\Fin_m$ into a $2$-generated subgroup $H$ of a group which is an extension of a group from $\Fin_m$ by a finite metabelian group. The construction
ensures that $H$ has property~$(T)$ whenever $G$ does, and the lower
bound on the Tarski number of $H$ then follows from Theorem~\ref{Tar_combined}(b).

Theorem~\ref{unbounded_2gen} has an interesting application to cogrowth and spectral radius (for the definition of cogrowth and spectral radius see \cite{CSGH}; note that both quantities are invariants of a marked group, that is, a group with a chosen finite generating set). Recall that the maximal possible cogrowth of an $m$-generated group is $2m-1$ and the maximal spectral radius is $1$. If $m>1$, an $m$-generated group is amenable if and only if its cogrowth is exactly $2m-1$ and if and only if the spectral radius is $1$. The formulas from \cite[Section IV]{CSGH} relating the Tarski number with the cogrowth and spectral radius of a group immediately imply the following:

\begin{Corollary}
For every $\eps> 0$ there exists a $2$-generated
non-amenable group $G$  such that every $2$-generated subgroup $\langle x,y\rangle$ of $G$ has cogrowth at least $3-\eps$ (with respect to the generating set $\{x,y\}$)
and spectral radius at least $1-\eps$.
\end{Corollary}
Note that since there are $2$-generated amenable groups (say, the lamplighter group $\Z/2\Z\wr \Z$) which are inductive limits of
$2$-generated non-amenable groups \cite{GH}, there are $2$-generated non-amenable groups with cogrowth arbitrarily close to $3$. But all previously known examples of such groups contain non-Abelian free subgroups and so they have $2$-generated subgroups with cogrowth $0$.
\skv

Theorem \ref{Tar_combined}(a) and Theorem \ref{t:ej} for $p=2$
yield another interesting corollary which is somewhat similar to the Bertrand postulate for prime numbers.

\begin{Corollary}[See Theorem \ref{Tar:estimates}]
\label{cor:n2n}
For every sufficiently large natural number $n$, there exists a group with Tarski number between $n$ and $2n$.
\end{Corollary}
\skv

\subsection{Explicit values}
The final problem we address in this paper is precise calculation of Tarski numbers.
Prior to this paper there were no examples of non-amenable
groups without free subgroups whose Tarski number had been determined.
In fact, no integer $>4$ was known to be the Tarski number of a group. We will show, using the cost of groups and random  forests on their Cayley graphs,
that there exist groups (in fact, a large class of groups) with Tarski number $5$ and (a large class of) groups with Tarski number~$6$. In what follows, $\mathcal C(G)$ will denote the cost of a group $G$. The definition of the cost of a group appears in \S~\ref{sec:6}.

\begin{Theoremintro}
\label{thm:Tarski56}
Let $G$ be a group generated by $3$ elements such that $\mathcal C(G)\geq 5/2$. The following hold:
\begin{itemize}
\item[(i)]  $\calT(G)\leq 6$. In particular, if $G$ is torsion, then $\calT(G)=6$.
\item[(ii)] Assume in addition that one of the $3$ generators of $G$ has infinite order. Then $\calT(G)\leq 5$.
In particular, if $G$ does not contain a non-Abelian free subgroup, then $\calT(G)=5$.
\end{itemize}
\end{Theoremintro}

Note that in \cite{Os}, Osin showed that for any integer $d\geq 2$ and any $\eps>0$ there exists
a $d$-generated torsion group $G$ with $\beta_1(G)\geq d-1-\eps$, where $\beta_1(G)$ denotes the
first $L^2$-Betti number of $G$.
By \cite[Corollary 3.23]{Gab1}, for every countable infinite group $G$ we have $ \mathcal C(G)\ge \beta_1(G)+1$. Thus, 
 torsion groups
satisfying the hypotheses of Theorem~\ref{thm:Tarski56}(i) (and therefore having Tarski number $6$) do exist. Moreover, one can construct such groups with very explicit presentations
(see Appendix~B). The last assertion of Theorem~\ref{thm:Tarski56}(i)
follows from Theorem~\ref{t:csgh}(i).

To our knowledge, groups satisfying the hypotheses of Theorem \ref{thm:Tarski56}(ii) do not appear in the literature. However, a slight modification of Osin's construction \cite{Os} yields such a group. 
For further details and very explicit presentations of groups with Tarski number $5$ see Appendix~B.

\begin{Remark}
One can define the notion of Tarski numbers for group actions (in complete analogy with the group case).
The question of which integers arise as Tarski numbers of group actions has been completely settled in a recent
paper of the second author \cite{Gi}, where it is proved that every integer $\ge 4$ is the Tarski number of a faithful transitive action of a free group.
\
\end{Remark}

\vskip .12cm
{\bf Organization.} In \S~\ref{sec:prelim} we introduce basic graph-theoretic terminology
and give a  graph-theoretic interpretation of paradoxical decompositions. In
\S~\ref{sec:amenablequotient} we prove Theorem~\ref{Tar_combined}.
In \S~\ref{sec:unbounded} we prove Theorem~\ref{unbounded_2gen}~and
discuss some related results.
In \S~\ref{sec:6} we prove Theorem~\ref{thm:Tarski56}. In Appendix~A we prove the equivalence of two definitions of Tarski numbers. In Appendix~B
we describe explicit constructions of groups with Tarski number $5$ and groups with Tarski number $6$. Finally, Appendix~C
contains a brief introduction to Golod-Shafarevich groups.

\vskip .12cm
{\bf Acknowledgments.} The authors would like to thank  Rostislav Grigorchuk,
Damien Gaboriau, Andrei Jaikin-Zapirain,  Wolfgang L\"uck, Russell Lyons, Nikolay Nikolov, Denis Osin, Dan Salajan and Andreas Thom for useful discussions and Narutaka Ozawa for posting question \cite{OS}.

\section{Preliminaries}
\label{sec:prelim}

\subsection{$k$-paradoxical decompositions}

It will be convenient to slightly generalize the notion of paradoxical decomposition
defined in the introduction (this generalization is also used, in particular, in \cite{RY}).

\begin{Definition}\rm Let $G$ be a group and $k\geq 2$ an integer. Suppose that
there exist finite subsets $S_1=\{g_{1,1},\ldots,g_{1,n_1}\},\ldots, S_k=\{g_{k,1},\ldots,g_{k,n_k}\}$ of $G$ and disjoint subsets $\{P_{ij}: 1\leq i\leq k,
1\leq j\leq n_i\}$ of $G$ such that for each $1\leq i\leq k$
we have $G=\bigcup_{j=1}^{n_i}P_{i,j}g_{i,j}$. Then we will say that $G$ admits
a {\it $k$-paradoxical decomposition with translating sets $S_1,\ldots, S_k$}.
The set $\cup_{i=1}^k S_i$ will be called the {\it total translating set}
of the decomposition.
\end{Definition}

Note that $2$-paradoxical decompositions are paradoxical decompositions in the usual sense.
Every $k$-paradoxical decomposition
with translating sets $S_1,...,S_k$ ``contains" a $2$-paradoxical decomposition with translating sets $S_1,S_2$. Conversely, given a $2$-paradoxical decomposition of a group $G$, there is
a simple way to construct a $k$-paradoxical decomposition of $G$ for arbitrarily large $k$
(see Lemma \ref{double_up}).
We will mostly use $2$-paradoxical decompositions, but $4$-paradoxical decompositions will
naturally arise in the proof of Theorem~\ref{Tar_combined}(b).
\vskip .12cm
The following result is obvious:

\begin{Remark}\label{1}
\label{pd:translation} If $G$ has a $k$-paradoxical decomposition with translating
sets $S_1,\ldots, S_k$, then $G$ also has a $k$-paradoxical decomposition with translating
sets $ S_1g_1,\ldots,  S_kg_k$ for any given $g_1,\ldots, g_k\in G$. In particular, we can always assume that $1\in S_i$ for each $i$.
\end{Remark}

Next we introduce some graph-theoretic terminology which is convenient
for dealing with paradoxical decompositions. We will mostly work with oriented graphs,
which will be allowed to have loops and multiple edges. In some cases edges of our
graphs will be colored and/or labeled.  The sets of vertices and edges of a graph $\Gamma$
will be denoted by $V(\Gamma)$ and $E(\Gamma)$, respectively.
If edges of $\Gamma$ are colored using colors $\{1,\ldots, k\}$, we denote by
$E_i(\Gamma)$ the set of edges of color $i$.

\begin{Definition}\rm $\empty$
\begin{itemize}
\item[(i)] Let $k$ be a positive integer. An oriented graph $\Gamma$ will
be called a {\it $k$-graph} if at each vertex of $\Gamma$ there are (precisely)
$k$ incoming edges and at most one outgoing edge.
\item[(ii)]  A $k$-graph $\Gamma$ will be called {\it evenly colored} if the edges
of $\Gamma$ are colored using $k$ colors and at each vertex of $\Gamma$ the $k$
incoming edges all have different colors.
\end{itemize}
\end{Definition}

Let $G$ be a group and $S$ a subset of $G$. The
Cayley graph $\Cay(G,S)$ is the {\bf oriented} graph with vertex set $G$ and a directed edge from $g$ to $gs$
for every $g\in G$ and $s\in S$. The edge $(g,gs)$ will be labeled by the element $s$.
We will also need the colored version of Cayley graphs.

\begin{Definition}\rm Let $S_1,\ldots, S_k$ be subsets of a group $G$. Let $E_i$
be the edge set of $\Cay(G,S_i)$, and define $\Cay(G,(S_1,\ldots, S_k))$ to be
the colored graph with vertex set $G$ and edge set $\sqcup_{i=1}^k E_i$ where
edges in $E_i$ are colored with color $i$. Note that if the sets $S_1,\ldots, S_k$ are not disjoint, the graph $\Cay(G,(S_1,\ldots, S_k))$ will have multiple edges, but there
will be at most one edge of a given color between any two vertices.
\end{Definition}

One can reformulate the notion of $k$-paradoxical decomposition
in terms of $k$-subgraphs of colored Cayley graphs as follows. Recall that a subgraph $\Gamma'$ of a graph $\Gamma$ is called \emph{spanning} if it contains all vertices of $\Gamma$.

\begin{Lemma}
\label{paradox} Let $S_1,\ldots, S_k$ be finite subsets of a group $G$.
The following are equivalent:
\begin{itemize}
\item[(i)] $G$ admits a $k$-paradoxical decomposition with translating sets
$S_1,\ldots, S_k$.
\item[(ii)] The colored Cayley graph $\Cay(G,(S_1,\ldots, S_k))$ contains
a spanning evenly colored $k$-subgraph.
\end{itemize}
\end{Lemma}

\begin{proof}
Assume that (ii) holds, and let $\Gamma$ be a spanning evenly colored $k$-subgraph
of $\Cay(G,(S_1,\ldots, S_k))$. For $1\leq i\leq k$ choose an ordering
$g_{i,1},\ldots, g_{i,n_i}$ of the elements of $S_i$.

For $1\leq i\leq k$ and $1\leq j\leq n_i$ let $P_{i,j}$ be the set of tail vertices of
all edges of $\Gamma$ which have color $i$ and label $g_{i,j}$, that is,
$$P_{i,j}=\{g\in G: (g,gg_{i,j})\in E_i(\Gamma)\}.$$
Since
every vertex of $\Gamma$ has at most one outgoing edge, the sets $P_{i,j}$ are disjoint.

On the other hand note that $P_{i,j}g_{i,j}$ is the set of head vertices of
all edges of $\Gamma$ of color $i$  labeled by $g_{i,j}$. Since every vertex of $\Gamma$ has exactly one incoming edge of any given color, for any $1\leq i\leq k$ we have $G=\sqcup_{j=1}^{n_i} P_{i,j}g_{i,j}$.
Therefore, $G$ has a $k$-paradoxical decomposition with translating sets
$S_1,\ldots, S_k$ and pieces $P_{i,j}$, so (i) holds.

Conversely, suppose that (i) holds. In the notations of the definition of a
$k$-paradoxical decomposition we can assume that the unions $\cup_{j=1}^{n_i} P_{i,j}g_{i,j}$
are disjoint (by making the sets $P_{i,j}$ smaller if needed). The rest of the proof
is completely analogous to the implication ``(ii)$\Rightarrow$ (i)''.
\end{proof}

\vskip .2cm

Given an oriented graph $\Gamma$ and a finite subset $A$ of $V(\Gamma)$, we put
\begin{align*}
&V^+_{\Gamma}(A)=\{v\in V(\Gamma): (a,v)\in E(\Gamma)\mbox{ for some }a\in A\};&\\
&V^-_{\Gamma}(A)=\{v\in V(\Gamma): (v,a)\in E(\Gamma)\mbox{ for some }a\in A\}.&
\end{align*}
In other words, $V^+_{\Gamma}(A)$ is the set of head vertices of all edges whose
tail vertex lies in $A$, and $V^-_{\Gamma}(A)$ is the set of tail vertices of all edges whose head vertex lies in $A$.

If in addition $E(\Gamma)$ is colored using colors $\{1,\ldots, k\}$,
we put
\begin{align*}
&V^{-,i}_{\Gamma}(A)=\{v\in V(\Gamma): (v,a)\in E_i(\Gamma)\mbox{ for some }a\in A\}&
\end{align*}
(recall that $E_i(\Gamma)$ is the set of edges of $\Gamma$ of color $i$).

A convenient tool for constructing $k$-subgraphs is the following version of the P. Hall marriage theorem.

\begin{Theorem}
\label{HallRado}
Let $\Gamma$ be a locally finite oriented graph and $k\geq 1$ an integer.
\begin{itemize}
\item[(i)] Assume that for every finite subset $A$ of $V(\Gamma)$ we have $|V^{-}_{\Gamma}(A)|\geq k|A|$. Then $\Gamma$ contains a spanning $k$-subgraph.
\item[(ii)] Suppose now that edges of $\Gamma$ are colored using colors $\{1,\ldots,k\}$.
Assume that for any finite subsets $A_1,\ldots, A_k$ of $V(\Gamma)$ we have
\begin{equation}
\label{HR_new}
|\cup_{i=1}^k V^{-,i}_{\Gamma}(A_i)|\geq \sum_{i=1}^k |A_i|.
\end{equation}
Then $\Gamma$  contains a spanning evenly colored $k$-subgraph.
\end{itemize}
\end{Theorem}
\begin{proof} A slight modification of the usual form of P. Hall theorem (see, for example, \cite[Lemma 5.8.25]{Sa}) asserts the following:
\begin{Lemma}
\label{HallRado_basic} Let $k\ge 1$ be an integer, $\calI$ and $S$ any sets
and $\{S_{\alpha}\}_{\alpha\in \calI}$ a collection of finite subsets of $S$. Suppose that for any finite subset $\calJ\subseteq \calI$ we have $|\cup_{\alpha\in \calJ} S_{\alpha}|\geq k|\calJ|$. Then there exist pairwise disjoint $k$-element subsets $\{X_{\alpha}\}_{\alpha\in \calI}$ such that $X_{\alpha}\subseteq S_{\alpha}$ for all $\alpha$.
\end{Lemma}
This immediately implies part (i).

In the setting of (ii), let $S=V(\Gamma)$, $\calI=V(\Gamma)\times\{1,\ldots, k\}$, and for $\alpha=(v,i)\in\calI$
put $S_{\alpha}=\{w\in V(\Gamma): (w,v)\in E_i(\Gamma)\}$, that is, $S_{\alpha}$ is the set of tail vertices of
edges of color $i$ in $\Gamma$ whose head vertex is $v$. Condition~\eqref{HR_new}
means precisely that Lemma~\ref{HallRado_basic} is applicable to this collection for $k=1$. If $\{x_{\alpha}\}_{\alpha\in \calI}$
is the resulting set of vertices, let $\Lambda$ be the spanning subgraph of $\Gamma$ with edges of the form
$(x_{(v,i)},v)\in E_i(\Gamma)$ for every $(v,i)\in\calI$. It is clear that $\Lambda$ is an evenly colored $k$-subgraph.
\end{proof}
Note that part (i) of Theorem \ref{HallRado} for arbitrary $k\ge 1$ can be easily deduced from part (ii):
starting with an uncolored graph $\Gamma$, we  consider the colored graph $\Gamma_k$ with $V(\Gamma_k)=V(\Gamma)$ and $E_i(\Gamma_k)=E(\Gamma)$ for $i\in\{1,\ldots,k\}$. Then the assumption $|V^{-}_{\Gamma}(A)|\geq k|A|$ ensures
that (ii) is applicable to $\Gamma_k$.

\section{Tarski numbers and extensions}
\label{sec:amenablequotient}

In this section we will prove Theorem~\ref{Tar_combined}.
The proofs of parts (a) and (b) of that theorem will be based on
Proposition~\ref{tarski_technical} below.
For the convenience of the reader we will restate all parts of the above theorem in this section.

Throughout the section $G$ will denote a fixed non-amenable
group and $H$ a subgroup of $G$. When $H$ is normal, $\rho\colon G\to G/H$ will denote the natural homomorphism.
Let $T$ be a right transversal of $H$ in $G$, that is, a subset of $G$ which contains precisely one element from each right coset of $H$.  Thus, there exist
unique maps $\pi_H\colon G\to H$ and $\pi_T\colon G\to T$ such that $g=\pi_H(g)\pi_T(g)$ for all
$g\in G$. We shall also assume that $1\in T$.

\begin{Proposition}
\label{tarski_technical}
Suppose that $G$ has a $k$-paradoxical decomposition with
translating sets $S_1,\ldots, S_k$ and assume that $1\in S_1$.
 Let $S=\cup_{i=1}^k S_i$.
Let $F$ be a subset of $T$, let $\Phi_i=\pi_T(FS_i^{-1})$ and
$\Phi=\pi_T(FS^{-1})=\cup_{i=1}^k \Phi_i$. Finally, let $S'_i=\Phi_i S_i F^{-1}\cap H$.
\begin{itemize}
\item[(i)] Suppose that $|\Phi|= |F|$. Then $H$ has
a $k$-paradoxical decomposition with translating sets $S'_1,\ldots, S'_k$.
Therefore, $\calT(H)\leq \sum_{i=1}^k |S'_i|$.
\item[(ii)] Suppose that $|\Phi|\leq \frac{k}{2}|F|$. Then $H$ has
a $2$-paradoxical decomposition with total translating set $\cup_{i=1}^k S'_i$.
Therefore, $\calT(H)\leq 2\sum_{i=1}^k |S'_i|$.
\end{itemize}
\end{Proposition}
\begin{proof}
(i) By Lemma~\ref{paradox}, the Cayley graph $\Cay(G,(S_1,\ldots, S_k))$
contains a spanning evenly colored $k$-subgraph $\Gamma$.
Let $\Gamma_0$ be the subgraph of $\Gamma$ with vertex set
$HFS^{-1}=H\Phi$ which contains an edge
$(g,g')$ of color $i$ if and only if $g'\in HF$ and $g\in HFS_i^{-1}=H\Phi_i$.
Note that by construction $\Gamma_0$ contains all edges of $\Gamma$
whose head vertex lies in $HF$.

Let $\Lambda$ be the quotient graph of $\Gamma_0$ in which we glue
two vertices $g$ and $g'$ of $\Gamma_0$ if and only if they have the same $H$-component,
that is, $\pi_H(g)=\pi_{H}(g')$. We do not make any edge identifications during this
process, so typically $\Lambda$ will have plenty of multiple edges.
We can naturally identify the vertex set of $\Lambda$ with $H$.

We claim that the graph $\Lambda$ satisfies the condition of Theorem~\ref{HallRado}(ii).
By construction, every vertex of $\Lambda$ has at most $|\Phi|$ outgoing edges.
Since $|\Phi|=|F|$, it suffices to show that for every $i\in\{1,\ldots, k\}$
and every finite subset $A$ of $H$ we have $|E^{-,i}_{\Lambda}(A)|\geq |F||A|$
where $E^{-,i}_{\Lambda}(A)$ is the set of edges of color $i$ whose head vertex lies in $A$.
The latter statement is clear since (again by construction) every vertex of $\Lambda$
has $|F|$ incoming edges of each color.

Thus, by Theorem~\ref{HallRado}(ii) $\Lambda$ contains a spanning evenly colored $k$-subgraph. By Lemma \ref{paradox}, to finish the proof of (i) it suffices to check that every edge of $\Lambda$ of color $i$ is labeled by an element of $S'_i$. Indeed, every edge $e=(h,h')\in E_i(\Lambda)$ comes from
an edge $(hf_i, h'f)\in E_i(\Gamma)$ for some $f\in F$ and $f_i\in \Phi_i$.
Hence $${\rm label}(e)=h^{-1}h'=f_i((hf_i)^{-1}h'f)f^{-1}\in \Phi_iS_iF^{-1}\cap H=S_i'.$$

\vskip .12cm

The proof of (ii) is almost identical except that we do not keep track of colors.
The same counting argument shows that
$|V^-_{\Lambda}(A)|\geq\frac{k|A||F|}{|\Phi|}\geq 2|A|$. Hence by Theorem~\ref{HallRado}(i),
$\Lambda$ contains a spanning $2$-subgraph, and by the above computation
all edges of $\Lambda$ are labeled by elements of $S'=\cup_{i=1}^k S'_i$.
Hence $\Cay(G,S')$ contains  a spanning $2$-subgraph, so $\Cay(G,(S',S'))$ contains
a spanning evenly colored $2$-subgraph.
\end{proof}

\begin{Theorem5}
Let $G$ be a non-amenable group and $H$ a subgroup of finite index in $G$.
Then $\calT(H)-2\leq [G:H](\calT(G)-2)$.
\end{Theorem5}
\begin{proof}
Choose a paradoxical decomposition of $G$ with translating sets $S_1$ and $S_2$
such that $|S_1|+|S_2|=\calT(G)$ and $1\in S_1\cap S_2$ (this is possible by Remark~\ref{pd:translation}),
 and apply Proposition~\ref{tarski_technical}
to that decomposition with $F=T$. Then
$\Phi_i=\Phi=T$ for each $i$, so hypotheses of Proposition~\ref{tarski_technical}(i) hold.

Note that for each $t'\in T$ and $g\in G$ there exists a unique $t\in T$
such that $t'gt^{-1}\in H$. Moreover, if $g=1$, then $t=t'$ and therefore
$t'gt^{-1}=1$. Hence for $i=1,2$ we have
$$|S'_i|=|\Phi_iS_iF^{-1}\cap H|=|TS_i T^{-1}\cap H|\le |T|(|S_i|-1)+1$$ (here we use the fact that each $S_i$ contains $1$).
Hence $\calT(H)\leq |S_1'|+|S_2'|\leq [G:H](|S_1|+|S_2|-2)+2=[G:H](\calT(G)-2)+2$.
\end{proof}

For the proof of Theorem~\ref{Tar_combined}(b) we will need the following two additional lemmas.

\begin{Lemma}
\label{lemma:variety} Let a variety $\mathcal V$ be as in Theorem~\ref{Tar_combined}(b). Then there exists a function $g\colon \mathbb{N}\to \mathbb{N}$ such that for every
$Z\in {\mathcal V}$ and every $n$-element subset $U\subseteq Z$ there exists a finite set $F\subseteq Z$ with $|F|\leq g(n)$, for which $|FU|\le 2|F|.$
\end{Lemma}

\begin{proof} Fix $n\in\dbN$. Let $V_n$ be the relatively free group of $\mathcal V$ of rank $n$. Let
$X=\{x_1,\ldots,x_n\}$ be the set of free generators of $V_n$. Since $V_n$ is amenable, by the F{\o}lner criterion \cite{Wa},
there exists a finite subset $F'\subseteq V_n$ such that $|F'X|\le 2|F'|$. We can assume that $F'$
has the smallest possible number of elements and define $g(n)=|F'|$. Given a group $Z
\in\mathcal V$ and an $n$-element subset $U\subseteq Z$, let $\phi$ be the homomorphism
$V_n\to Z$ which sends $x_i$ to the corresponding element of $U$. Let $\prec$ be a total
right-invariant order on $V_n$. Let
$$R=\{r\in F': \mbox{ there exists } f\in F'\mbox{ with }f\prec r
\mbox{ and }\phi(f)=\phi(r)\}.$$
Clearly $R$ is a  proper subset of $F'$, so the minimality of $F'$ implies $|RX|>2|R|$. Note that
$\phi(F')=\phi(F'\setminus R)$ and $\phi$ is injective on $F'\setminus R$.
Since $\prec$ is right-invariant, for every $x\in X$ and $r\in R$ there exists $f\in F'$ such that $fx\prec rx$ and $\phi(fx)=\phi(rx)$. Hence $\phi(F')U=\phi(F'X)=\phi(F'X\setminus RX)$, so
$$\begin{array}{l}|\phi(F')U|\le |\phi(F'X\setminus RX)| \le |F'X|-|RX|\le 2|F'|-2|R|\\=2|F'\setminus R|=2|\phi(F'\setminus R)|= 2|\phi(F')|.\end{array}$$ Thus we can take $F=\phi(F')$.
\end{proof}

\begin{Lemma}
\label{double_up}
Suppose that a group $G$ has a $k$-paradoxical decomposition with
translating sets $S_1,\ldots, S_k$ and an $l$-paradoxical decomposition with
translating sets $T_1,\ldots, T_l$. Then $G$ has a $kl$-paradoxical decomposition with
translating sets $\{S_i T_j\}$.
\end{Lemma}

\begin{proof}
By Lemma \ref{paradox}, the Cayley graphs $\Cay(G,(S_1,\dots,S_k))$ and $\Cay(G,(T_1,\dots,T_l))$ have spanning
evenly colored $k$-subgraph $\Gamma_k$ and $l$-subgraph $\Gamma_l$, respectively. Let $\Gamma=\Gamma_k\Gamma_l$ be the graph with vertex set $G$ and edge set
$$E(\Gamma)=\{(g,g')\in G\times G:  \mbox{there is } x\in G \mbox{ s.t. } (g,x)\in E(\Gamma_k)
\mbox{ and } (x,g')\in E(\Gamma_l) \}.$$
In other words, edges of $\Gamma$ are (oriented) paths of length $2$ where the first
edge of the path lies in $\Gamma_k$ and the second lies in $\Gamma_l$. Using the colorings
of $\Gamma_k$ and $\Gamma_l$, we can naturally color $E(\Gamma)$ with $kl$ colors. It
is clear that $\Gamma$ will become a spanning evenly colored $kl$-graph, representing a $kl$-paradoxical decomposition with translating sets $\{S_i T_j\}$.
\end{proof}

\begin{Theorem7}
Let $\mathcal{V}$ be a variety of groups where all groups are amenable and relatively free groups are right orderable. Then there exists a function $f\colon \mathbb{N}\to \mathbb{N}$ (depending only on $\mathcal{V}$)
with the following property: if a non-amenable group $G$ has a normal subgroup $H$ such that
$G/H\in \mathcal{V}$, then $\calT(H)\leq f(\calT(G))$.
\end{Theorem7}

\begin{proof}
Let $n=\calT(G)$ and $Z=G/H$. Recall that $\rho\colon G\to Z$ is the natural homomorphism. By Lemma~\ref{double_up}, $G$ has a $4$-paradoxical decomposition
with translating sets $S_1,S_2,S_3,S_4$ such that $1\in S_1$ and
  $$\sum_{i=1}^4 |S_i|\leq n^2.$$

Let $S=\cup_{i=1}^4 S_i$. By
Lemma~\ref{lemma:variety} applied to the set $U=\rho(S^{-1})$, there exists $\overline F\subseteq Z$
with $|\overline F|\leq g(n^2)$ such that $|\overline F U|\leq 2|\overline F|$. Let $F=\rho^{-1}(\overline F)\cap T$.
Then $|F|=|\overline F|$ and $\rho$ bijectively maps $\Phi=\pi_T(FS^{-1})$ onto $\overline F U$.
Therefore, hypotheses of Proposition~\ref{tarski_technical}(ii) hold, and we deduce
that $H$ has a $2$-paradoxical decomposition with total translating set
$\cup_{i=1}^4 S_i'$ (where $S_i'$ are defined as in Proposition~\ref{tarski_technical}).
Clearly, $|S_i'|\leq |S_i||F||\Phi|=|S_i||F||\overline F U|\leq 2|S_i||F|^2$, and therefore
$\calT(H)\leq 2\sum_{i=1}^4 2|S_i||F|^2\leq 4n^2 g(n^2)^2$.
\end{proof}

For the proof of Theorem~\ref{Tar_combined}(c) we will need the following lemma.

\begin{Lemma}\label{Folner}
Assume that $H$ is normal and amenable.
Suppose that $U_1,U_2\sub G$ are finite subsets such that for every pair of finite subsets $F_1,F_2\sub G$
we have $|\bigcup_{i=1}^{2}F_iU_i|\ge\sum_{i=1}^2 |F_i|$. Let $U_i'=\rho(U_i)$. Then for every pair of finite subsets $F_1',F_2'\sub G/H$ we have $|\bigcup_{i=1}^{2}F_i'U_i'|\ge\sum_{i=1}^2 |F_i'|$.
\end{Lemma}

\begin{proof}
Let $U=U_1\cup U_2$. Let $\psi \colon G/H\to T$ be the unique map such that
$\rho\psi (gH)=gH$ for all $g\in G$. Note that $\psi$ is a bijection and
$\psi\rho(g)=\pi_T(g)$ for all $g\in G$.

Fix $\eps>0$.
Let $F_1',F_2'\sub G/H$ be finite sets, let $F_i''=\psi(F_i')$ and $F''=F_1''\cup F_2''$. Let $U_H=\pi_H(F''U)$.
Since $U_H\sub H$ is a finite subset of the amenable group $H$, by F{\o}lner's criterion,  there exists a finite set $F_H\sub H$ such that $|F_H U_H|<(1+\eps)|F_H|$.
Define $F_i=F_HF_i''\sub G$. Since $F_H\sub H$ and $F_i''\sub T$, we have
$|F_i|=|F_H||F_i''|=|F_H||F_i'|$.

Note that $F_iU_i\sub (F_HU_H)\psi(F_i'U_i')$. Indeed,
\begin{multline*}
F_iU_i=F_H F_i'' U_i\sub F_H \pi_H(F_i'' U_i)\pi_T(F_i'' U_i)\sub
F_H U_H\cdot \psi\rho(F_i''U_i)\\
= F_H U_H \psi(\rho(F_i'')\rho(U_i))=F_H U_H\psi(F_i'U_i').
\end{multline*}
Therefore, $\bigcup_{i=1}^{2}F_iU_i\sub \bigcup_{i=1}^{2} F_H U_H \psi(F_i'U_i')=F_HU_H\psi(\bigcup_{i=1}^{2}F_i'U_i')$.
Hence, $|\bigcup_{i=1}^{2}F_iU_i|$ $\le |F_HU_H||\bigcup_{i=1}^{2}F_i'U_i'|$. Since $|\bigcup_{i=1}^{2}F_iU_i|\ge\sum_{i=1}^2 |F_i|$
by the hypotheses of the theorem, we get $|F_HU_H||\bigcup_{i=1}^{2}F_i'U_i'|\ge \sum_{i=1}^2 |F_H||F_i'|$. Hence $(1+\eps)|F_H||\bigcup_{i=1}^{2}F_i'U_i'|\ge \sum_{i=1}^2 |F_H||F_i'|$ and $(1+\eps)|\bigcup_{i=1}^{2}F_i'U_i'|\ge \sum_{i=1}^2 |F_i'|$. Since the inequality holds for every $\eps>0$, we conclude that $|\bigcup_{i=1}^{2}F_i'U_i'|\ge \sum_{i=1}^2 |F_i'|$.
\end{proof}

\begin{Theorem8}
Let $G$ be a non-amenable group with an amenable normal subgroup $H$.
Then $\calT(G/H)=\calT(G).$
\end{Theorem8}

\begin{proof}
Let $n=\calT(G)$ and choose a paradoxical decomposition of $G$ with $|S_1|+|S_2|=n$.
By Lemma~\ref{paradox}, the Cayley graph $\Cay(G,(S_1,S_2))$
contains a spanning evenly colored $2$-subgraph $\Gamma$. In particular, for every pair of finite subsets $F_1,F_2\sub G$ we have $|\bigcup_{i=1}^{2}F_iS_i^{-1}|\ge\sum_{i=1}^2 |F_i|$. By Lemma~\ref{Folner} the Cayley graph $\Gamma' =$ $\Cay(G/H,(\rho(S_1),\rho(S_2)))$ satisfies the condition of Theorem~\ref{HallRado}(ii) with $k=2$. Thus, $\Gamma'$ contains a spanning evenly colored $2$-subgraph. Therefore $\calT(G/H)\le |\rho(S_1)|+|\rho(S_2)|\le\calT(G)$.
\end{proof}

To prove Theorem~\ref{Tar_combined}(d) we will use the following variation of Lemma~\ref{Folner}.

\begin{Lemma}\label{U_GxG}
Let $G=H_1\times H_2$, and let $U\sub G$ be a finite subset such that for each finite subset $F\sub G$
we have $|FU|\ge 2|F|$. Let $U_1=\pi_1(U)$ and $U_2=\pi_2(U)$ where $\pi_1$ and $\pi_2$ are the projections onto $H_1$ and $H_2$ respectively.
Then for some $i\in\{1,2\}$ for any finite subset $F_i\sub H_i$ we have $|F_i(U_i)^2|\ge 2|F_i|$.
\end{Lemma}

\begin{proof}
If $|F_1U_1|\ge \sqrt{2}|F_1|$ for each finite subset $F_1\sub H_1$,
then replacing the finite subset $F_1$ by $F_1U_1$, we get $|F_1(U_1)^2|=|(F_1U_1)U_1|\ge \sqrt{2}|F_1U_1|\ge 2|F_1|$ and we are done.
Otherwise, fix $F_1\sub H_1$ such that $|F_1U_1|<\sqrt{2}|F_1|$.
Given a finite subset $F_2\sub H_2$, let $F=F_1\times F_2$. Note that $U\sub U_1\times U_2$ implies that $$(F_1U_1)\times (F_2U_2)=(F_1\times F_2)(U_1\times U_2)\supseteq FU.$$
Therefore $|F_1U_1||F_2U_2|\ge |FU|\ge 2|F|=2|F_1||F_2|$. Hence $|F_2U_2|\ge 2 \frac{|F_1|}{|F_1U_1|} |F_2|\ge \sqrt{2}|F_2|$. As before replacing $F_2$ by $F_2U_2$ yields the required inequality.
\end{proof}

\begin{Theorem9}
Let $G=H_1\times H_2$ be a non amenable group. Then $\min\{\calT(H_1),\calT(H_2)\}\leq 2(\calT(G)-1)^2$.
\end{Theorem9}

\begin{proof}
Let $n=\calT(G)$ and choose a paradoxical decomposition of $G$ with $|S_1| + |S_2| = n$ and $1 \in S_1\cap S_2$.
By Lemma~\ref{paradox}, the Cayley graph $\Cay(G,(S_1,S_2))$
contains a spanning evenly colored $2$-subgraph $\Gamma$. In particular, for $U=S_1\cup S_2$ and every finite $F\sub G$ we have $|FU^{-1}|\ge 2|F|$. Let $U_1=\pi_1(U)$, $U_2=\pi_2(U)$. By Lemma~\ref{U_GxG}, for $i=1$ or $i=2$ the Cayley graph $\Cay(H_i,(U_i)^2)$ satisfies the hypotheses of Theorem~\ref{HallRado}(i) with $k=2$. Thus, $\Cay(H_i,(U_i)^2)$ contains a spanning $2$-subgraph, so $\Cay(H_i,((U_i)^2,(U_i)^2))$ contains a spanning evenly colored $2$-subgraph. Therefore $\calT(H_i)\le 2|U_i|^2\le 2|U|^2\le 2(n-1)^2$.
\end{proof}

\section{Unbounded Tarski numbers}
\label{sec:unbounded}

In this section we will prove Theorem~\ref{unbounded_2gen} and discuss Theorem~\ref{t:ej}
and its corollaries.

\subsection{Lower bound on Tarski numbers}\label{s:l}

We start with a simple lemma which can be used to bound Tarski numbers from below.
Part (a) is an observation of Ozawa from \cite{OS} and part (b) is a natural
generalization of Theorem~\ref{t:csgh}(i).

\begin{Lemma}
\label{obs:Ozawa}
Assume that $G$ has a paradoxical decomposition with translating
sets $S_1$ and $S_2$. Then
\begin{itemize}
\item[(a)] The subgroup generated by $S_1\cup S_2$ is non-amenable.
\item[(b)] The subgroup generated by $S_i$ is infinite for $i=1,2$.
\end{itemize}
\end{Lemma}

\begin{proof}
(a) Let $H$ be the subgroup generated by $S_1\cup S_2$. Intersecting each set in the given paradoxical decomposition of $G$ with $H$ gives a paradoxical decomposition of $H$ with the same sets of translating elements.

(b) By Lemma \ref{paradox}, the colored Cayley graph $\Cay(G,(S_1,S_2))$ has a spanning evenly colored $2$-subgraph $\Gamma$.
Choose any edge of $\Gamma$ of color $2$, and let $g_0$ be the tail vertex of that edge. Let $g_1,g_2,\ldots$ be the
sequence of elements of $G$ defined by the condition that $(g_{i+1},g_{i})$ is an edge of color $1$ in $\Gamma$ for all $i\geq 0$
(such a sequence is unique since each vertex has a unique incoming edge of color $1$ in $\Gamma$).

We claim that all elements $g_i$ are distinct. Indeed, suppose that $g_i=g_j$ for $i<j$, and assume that $i$ and $j$ are
the smallest with this property. If $i>0$, the vertex $g_i$ would have two outgoing edges in $\Gamma$, namely
$(g_{i},g_{i-1})$ and $(g_{j},g_{j-1})$, a contradiction. If $i=0$, we get a contradiction with the assumption that
$g_0$ has an outgoing edge of color $2$.

By construction, all elements $g_0^{-1}g_i$ lie in the subgroup generated by $S_1$, so this subgroup must be infinite.
By the same argument, the subgroup generated by $S_2$ is infinite.
\end{proof}

Recall that $\Amen_m$ (resp. $\Fin_m$) denotes the class of groups in which all $m$-generated subgroups are
amenable (resp. finite).  Combining Lemma~\ref{obs:Ozawa}~and~Remark~\ref{pd:translation}, we
deduce the following statement:

\begin{Corollary} \label{c:2}
 Let $G$ be a non-amenable group and $m\in\dbN$. The following hold:
\begin{itemize}
\item[(i)] If $G$ belongs to $\Amen_m$, then $\calT(G)\geq m+3$.
\item[(ii)] If $G$ belongs to $\Fin_m$, then $\calT(G)\geq 2m+4$.
\end{itemize}
\end{Corollary}

In particular, to prove that groups with unbounded Tarski numbers exist,
it suffices to know that $\Amen_m$ contains non-amenable groups for every $m$.
As noticed in~\cite{OS}, already $\Fin_m$ contains non-amenable groups for every $m$ --
this follows from the next two theorems on Golod-Shafarevich groups:

\begin{itemize}
\item[(i)] (see \cite{EJ}) Every Golod-Shafarevich group has an infinite quotient with property $(T)$.
In particular, every Golod-Shafarevich group is non-amenable.
\item[(ii)] (see  \cite{Go}, \cite[Theorem~3.3]{Er2}) For every $m$ there exists an $(m+1)$-generated Golod-Shafarevich
group in $\Fin_m$.
\end{itemize}

Since the class $\Fin_m$ is obviously closed under taking quotients, (i) and (ii) actually
yield a stronger corollary, which will be needed to prove Theorem~\ref{unbounded_2gen}:

\begin{Corollary}
\label{lem:finT}
For every $m\in\dbN$ there exists an infinite property $(T)$ group in $\Fin_m$.
\end{Corollary}

\subsection{Proof of Theorem~\ref{unbounded_2gen}}
Recall the formulation of the theorem:

\begin{Theorem2} The set of Tarski numbers of $2$-generated non-amenable groups is infinite.
Moreover, the set of Tarski numbers of $2$-generated infinite groups with property $(T)$ is infinite.
\end{Theorem2}

The proof of Theorem \ref{unbounded_2gen} is based on Theorem~\ref{t:hh} below which is proved using results and ideas from the classical paper by Bernhard Neumann and Hanna Neumann \cite{HH}.

\begin{Theorem}\label{t:hh} Let $G$ be a finitely generated group. The following hold:
\begin{itemize}
\item[(a)] The derived subgroup $[G,G]$ of $G$ embeds into a $2$-generated subgroup $H$ of a wreath product $G\wr C_n$ for a sufficiently large $n\in \dbN$ where $C_n$ is the cyclic group of order $n$. Moreover, $H$ contains the derived subgroup $[G^n,G^n]=[G,G]^n$ of the base group of the wreath product.

\item[(b)] Assume in addition that $G$ is torsion. Then $G$ embeds into a $2$-generated subgroup $H$ of a group $L$
which is an extension of a finite direct power $G^n$ of $G$ (for some $n\in \dbN$) by a finite metabelian group.

\item[(c)] Assume in addition that $G$ has property $(T)$. Then in both (a) and (b) $H$ has property $(T)$.
\end{itemize}
\end{Theorem}

\begin{proof} (a) Let $d$ be the number of generators of $G$. Take any $n>2^{2^{2d}}$. Let $z$ be a generator of $C_n$. For an element $g\in G$ let $\delta(g)\colon C_n\to G$ be the function given by $\delta(g)(1)=g$ and $\delta(h)(c)=1$ for $c\neq 1$.
Let $X=\{x_1,\ldots,x_d\}$ be a generating set of $G$, and define the function $a\colon C_n\to G$ by

$$a(z^k)=\left\{
\begin{array}{ll}
x_i&\mbox{ if }k=2^{2^i}\mbox{ for some }1\leq i\leq d,\\
1&\mbox{ otherwise.}
\end{array}
\right.$$

Let $H$ be the subgroup of $G\wr C_n$ generated by $a$ and $z$. Then it is easy to see
that $[a^{z^{-2^{2^i}}},a^{z^{-2^{2^j}}}]=\delta([x_i,x_j])$. For every word $w$ in $x_i^{\pm 1}$,
the function $\delta([x_i,x_j]^w)$ can be obtained from  $\delta([x_i,x_j])$ by
conjugation by a product of elements of the form $a^{z^m}$. Thus, $H$ contains
all functions of the form $\delta([x_i,x_j]^w)$, and clearly these functions
generate the subgroup $G_1$ of $G^n$ consisting of all functions $f$ with $f(1)\in [G,G], f(c)=1$ if $c\ne 1$.  Since $z\in H$, the subgroup $H$ contains all conjugates $G_1^{z^m}$, hence it contains the derived subgroup of the base group of $G\wr C_n$.

(b) Again let $d$ be the number of generators of $G$. Since $G$ is torsion, by \cite[Lemma~4.1]{HH} $G$ embeds into the derived subgroup of the $(k+1)$-generated group $W=G\wr C_n$ where $n$ is the least common multiple of the orders of the generators of $G$.
By (a), $[W,W]$ embeds into a $2$-generated subgroup of the group $L=W \wr C_m$ for some $m$, and the proof is complete.

(c) We will prove the result in the setting of (a); the proof in the setting of (b) is analogous. Since $G$ has property $(T)$, the
abelianization $G/[G,G]$ is finite. Therefore $[G,G]^n$ is a finite index subgroup of $G\wr C_n$, so in particular $[G,G]^n$ has finite index in $H$. Since the direct product of two groups with property $(T)$
has property $(T)$ and property $(T)$ is preserved by finite index subgroups and overgroups (see \cite{BH}), we conclude that $H$ has property
$(T)$.
\end{proof}

Now let $f_{1}\colon \dbN\to\dbN$ and $f_{2}\colon \dbN\to\dbN$ be the functions from Theorem~\ref{Tar_combined}(b)
corresponding to the varieties of Abelian groups and metabelian groups, respectively, and define
$g_1,g_2\colon \dbN\to\dbN$ by $g_i(x)=\min \{t: f_i(t)\geq x\}$. Clearly, $g_i(x)\to\infty$ as $x\to\infty$.

\begin{Corollary}
\label{embed_2gen}
Let $G$ be any finitely generated non-amenable group from $\Fin_m$. In the above notations the following hold:
\begin{itemize}
\item[(a)] $G$ embeds into a $2$-generated group $H$ with $\calT(H)\geq g_2(2m+4)$.
\item[(b)] $[G,G]$ embeds into a $2$-generated group $H$ with $\calT(H)\geq g_1(2m+4)$.
\item[(c)] If $G$ has property $(T)$, then in both (a) and (b) $H$ also has property $(T)$.
\end{itemize}
\end{Corollary}
\begin{proof} This follows directly from Theorem~\ref{t:hh}, Lemma~\ref{obs:Ozawa}(b), the obvious
fact that if $G$ lies in $\Fin_m$, then any finite direct power of $G$ lies in $\Fin_m$,
and the fact that free metabelian groups and free Abelian groups are right orderable (Remark \ref{r:1}(ii)).
\end{proof}

Theorem~\ref{unbounded_2gen} now follows immediately from Corollary~\ref{embed_2gen} (we can use either (a) or (b)
combined with (c)) and Corollary~\ref{lem:finT}.

\subsection{The Bertrand-type property of Tarski numbers}

As we already stated, Golod-Shafarevich groups are always non-amenable by \cite{EJ}. Moreover, if $G$
is a Golod-Shafarevich group with respect to a prime $p$, the image of $G$ in its
pro-$p$ completion (which is a residually-$p$ group) is non-amenable. Therefore,
Theorem~\ref{t:ej} is a corollary of the following result:

\begin{Proposition}
\label{cor:GS1}
Let $p$ be a prime and $G$ a Golod-Shafarevich group with respect to $p$. Then there is
a quotient $H$ of $G$ which is also Golod-Shafarevich (with respect to $p$), $p$-torsion and
such that for every $m\in\dbN$ there is a finite index subgroup $H_m$ of $H$ which lies in $\Fin_m$.
\end{Proposition}

Proposition~\ref{cor:GS1} follows immediately from the proof (but not quite from the statement) of
\cite[Lemma~8.8]{EJ2}. For completeness, we will present a self-contained proof
of Proposition~\ref{cor:GS1} in Appendix~C, where we will also define Golod-Shafarevich groups
and some related notions.

Note that if $p\geq 67$, one can deduce Theorem~\ref{t:ej} from Proposition~\ref{cor:GS1} without
using non-amenability of arbitrary Golod-Shafarevich groups (but using the fact that a Golod-Shafarevich group
with respect to $p$ has infinite pro-$p$ completion). Indeed, in  that case there exists a
Golod-Shafarevich group $G$ with respect to $p$ with property $(T)$ (see \cite[Theorem~12.1]{Er2}), so all infinite quotients of
$G$ are automatically non-amenable.

We do not know the answer to the following question, which can be thought of as a ``dual'' version of Theorem~\ref{t:ej}.

\begin{Problem}
\label{Tar_quotient}
Does there exist a sequence of finitely generated non-amenable groups $\{G_n\}_{n\in\dbN}$ such that
$G_{n+1}$ is a quotient of $G_n$ for each $n$ and $\calT(G_n)\to \infty$ as $n\to\infty$?
\end{Problem}

Note that while by the above argument the group $H$ in Theorem~\ref{t:ej} (and its subgroups of finite index) can be chosen to have property $(T)$, groups $G_n$ satisfying the hypotheses of Problem~\ref{Tar_quotient}
(if they exist) cannot have property $(T)$. Indeed, the inductive limit $G_\infty$ of a sequence $\{G_n\}$ of such groups cannot have a finite Tarski number. Hence $G_\infty$ is amenable. Suppose that one of the groups $G_n$ has property $(T)$. Then $G_\infty$ also has property $(T)$. Therefore $G_\infty$ is finite, so  $G_\infty$ has a finite presentation. The relations of that presentation must follow from the relations of one of the groups $G_n$. Therefore $G_n$ is a homomorphic image of $G_\infty$, whence $G_n$ is finite, a contradiction.

We conclude this section with the proof of Corollary~\ref{cor:n2n} restated below:

\begin{Theorem}
\label{Tar:estimates}
For every sufficiently large $n$ there exists a group $H$
with $n\leq \calT(H)\leq 2n$.
\end{Theorem}
\begin{proof}
Let $H$ be a group satisfying the conclusion of Theorem~\ref{t:ej}
for $p=2$. Then $H$ has a descending chain of normal subgroups $H=H_1\supset H_2\supset\ldots$
such that $[H_i:H_{i+1}]=2$ for all $i$ and $\calT(H_i)\to\infty$. Thus,
Theorem~\ref{Tar:estimates} follows from Theorem~\ref{Tar_combined}(a).
\end{proof}

\section{Tarski numbers, cost of groups and random forests}
\label{sec:6}

\subsection{Cost and random forests}\label{sec:87}

Recall the definition of the cost of a countable group $G$ (see \cite{Gab}).
Let $G$ be a countable group.
Let $(X,\mu)$ be a Borel probability measure space and let $G\acts X$ be an almost surely free (i.e., free outside a subset of measure $0$) left Borel action of $G$ on $X$ preserving $\mu$. Let $\Phi=\{\phi_i, i=1, 2,\ldots\}$ be a countable collection of Borel bijections between Borel subsets $A_i$ and $B_i$ of $X$ such that for every $x\in A_i$ the point $\phi_i(x)$ belongs to the orbit $G\cdot x$.
Then we can construct a graph with vertex set $X$ and edges connecting each $x\in A_i$ with $\phi_i(x)$. If connected components of that graph are (almost surely) the orbits of $G$, then we call $\Phi=\{\phi_i\}$ a \emph{graphing} of the action $G\acts X$. The \emph{cost} of the graphing
$\Phi$, denoted by $\cost(\Phi)$, is the sum of measures $\sum \mu(A_i)$. The \emph{cost of the action} $G\acts X$, denoted by $\cost(G\acts X)$, is the infimum of costs of all graphings.  The infimum of the costs of all such actions of $G$
is called the \emph{cost of the group} $G$ and denoted by $\mathcal C(G)$.
For any finite group $G$ we have  $\mathcal C(G)=1-\frac{1}{|G|}$, whereas if $G$ is countably infinite, $\mathcal C(G)\ge 1$ (see, for example, \cite[Section~29]{KM}).
If all actions of $G$ have the same cost, then $G$ is said to have {\it fixed price}. It is one of the outstanding open problems, called the Fixed Price problem, whether every countable group $G$ has fixed price.

The cost of a group $G$ is closely related to the degree of certain $G$-invariant random spanning forests on its 
(unoriented) Cayley graphs. Let $G$ be a group and $S$ a finite generating set of $G$. Define the 
{\it unoriented Cayley graph $\Cayu(G,S\cup S^{-1})$} to be the graph obtained from $\Cay(G,S\cup S^{-1})$
by replacing every pair of mutually opposite edges $(u,v)$ and $(v,u)$ by one unoriented edge $\{u,v\}$.
An edge $\{g,gs\}$ of $\Cayu(G,S\cup S^{-1})$ will be labeled by the formal symbol $s^{\pm 1}$ (regardless
of whether $s^{-1}=s$ in $G$ or not)

A $G$-invariant random spanning forest $\mu$ on $\Gamma=\Cayu(G,S\cup S^{-1})$ is a Borel probability measure on spanning subgraphs of $\Gamma$  which is supported on forests and $G$-invariant in the following sense. Let $\Sigma_{\Gamma}$ be the set of all spanning subgraphs of $\Gamma$ which can be identified with the space $\{0,1\}^{E(\Gamma)}$ with the product topology.
Let $\mu$ be a Borel probability measure on $\Sigma_{\Gamma}$. We say that $\mu$ is supported on forests
if $\mu(\{\Lambda \in \Sigma_{\Gamma}: \Lambda \mbox{ is a forest}\})=1$.
The natural left multiplication action of $G$ on $\Gamma$ induces the corresponding action
of $G$ on $\Sigma_{\Gamma}$. We say that $\mu$ is $G$-invariant if it is invariant under this action. Since the action of $G$ on $V(\Gamma)$ is transitive, if $\mu$ is $G$-invariant, all the vertices of $\Gamma$ have the same expected degree, that is,
$\deg_{\mu}(g)=\deg_{\mu}(e)$ for all $g\in G$, where $\deg_{\mu}(g)$ denotes the expected degree of $g$. This degree
is called the degree of the $G$-invariant random forest $\mu$ and denoted by $\deg(\mu)$.

There are several standard constructions of $G$-invariant random spanning forests on Cayley graphs including
the {\it free minimal spanning forest} 
(we refer the reader to \cite{Th} for the definition which will not be important to us).

The following theorem appears as Proposition~4.1 in \cite{EM}, where the result is attributed to Lyons
(see also \cite{L} for related results).

\begin{Theorem}
\label{spanning6}
Let $G$ be a group generated by a finite set $S$. 
Then the degree of the free minimal spanning forest on the (unoriented) Cayley graph $\Cayu (G, S\cup S\iv)$ is at least $2\mathcal C(G)$.
\end{Theorem}

Theorem~\ref{spanning6} will be used to prove Theorem~\ref{thm:Tarski56}(i). In order to prove Theorem~\ref{thm:Tarski56}(ii), we need a certain variation of the free minimal spanning forest constructed in a recent paper of Thom:

\begin{Theorem}[Theorem 3, \cite{Th}]\label{spanning5}
Let $G$ be a group generated by a finite set $S$ and let $a\in S$ be an element of infinite order.
Then there exists a $G$-invariant random spanning forest $\mu$ on $\Cayu(G,S\cup S^{-1})$ such that
\begin{enumerate}
\item $\mu$-a.s. the forest contains all the edges of $\Cayu(G,S\cup S^{-1})$ labeled by $a^{\pm 1}$;
\item $\deg(\mu)\ge 2\calC(G)$.
\end{enumerate}
\end{Theorem}

\subsection{Upper bounds on Tarski numbers in terms of random forests and the proof of Theorem~\ref{thm:Tarski56}}

The following proposition shows how the Tarski number of a group $G$ can be bounded above in terms
of the degree of a $G$-invariant random spanning forest on a Cayley graph of $G$.

\begin{Proposition}
\label{prop:bound_randomforest}
Let $G$ be a finitely generated group, $S$ a finite generating set of $G$, and
let $\Gamma=\Cay_{uo}(G,S\cup S^{-1})$. 
Let $\mu$ be a $G$-invariant random spanning forest
on $\Gamma$ and let $\delta=\deg(\mu)$. Let $T=S\cup\{1\}$.
The following hold:
\begin{itemize}
\item[(a)] For any finite subset $A$ of $G$ we have $|AT^{-1}|\geq (\delta-|S|)|A|$.
\item[(b)] For any finite subset $A$ of $G$ we have $|A(T\cup T^{-1})|\geq (\delta-1)|A|$.
\item[(c)] If $\delta-|S|\geq 2$ and $s_1,s_2$ are distinct elements of $S$,
then $G$ has a paradoxical decomposition with translating sets $T_1=T\setminus\{s_1\}$ and
$T_2=T\setminus\{s_2\}$. In particular, $$\calT(G)\leq 2|S|.$$
\item[(d)] If $\delta\geq 3$ and $s_1,s_2$ are distinct elements of $S$,
then $G$ has a paradoxical decomposition with translating sets $T\cup T^{-1}\setminus\{s_1\}$ and
$T\cup T^{-1}\setminus\{s_2\}$. In particular, $$\calT(G)\leq 4|S|.$$
\end{itemize}
Assume now that $S$ contains an element $a$ of infinite order and that $\mu$
almost surely contains all edges of $\Gamma$ labeled by $a^{\pm 1}$. Let
$T_1=T\setminus\{a\}$ and $T_2=\{1,a\}$. The following hold:
\begin{itemize}
\item[(e)] For any finite subsets $A$ and $B$ of $G$ we have
$$|AT_1^{-1}\cup BT_2^{-1}|\geq |A|+(\delta-|S|-1)|B|.$$
\item[(f)] For any finite subsets $A$ and $B$ of $G$ we have
$$|A(T_1\cup T_1^{-1})\cup B T_2^{-1}|\geq (\delta-3)|A|+|B|.$$
\item[(g)] If $\delta-|S|\geq 2$, then $G$ has a paradoxical decomposition
with translating sets $T_1$ and $T_2$.
\item[(h)] If $\delta\geq 4$, then $G$ has a paradoxical decomposition with
translating sets $T_1\cup T_1^{-1}$ and $T_2$.
\end{itemize}
\end{Proposition}

Before proving Proposition~\ref{prop:bound_randomforest}, we explain how Theorem~\ref{thm:Tarski56}
follows from it. We first recall the statement of the theorem:

\begin{Theorem10}
Let $G$ be a group generated by $3$ elements such that $\mathcal C(G)\geq 5/2$. The following hold:
\begin{itemize}
\item[(i)]  $\calT(G)\leq 6$. In particular, if $G$ is torsion, then $\calT(G)=6$.
\item[(ii)] Assume in addition that one of the $3$ generators of $G$ has infinite order. Then $\calT(G)\leq 5$.
In particular, if $G$ does not contain a non-Abelian free subgroup, then $\calT(G)=5$.
\end{itemize}
\end{Theorem10}

\begin{proof} Let $S$ be a generating set of $G$ with $3$ elements and
let $\Gamma=\Cayu(G,S\cup S^{-1})$ be the associated unoriented Cayley graph.
In the setting of Theorem~\ref{thm:Tarski56}(i), let $\mu$ be the free minimal spanning forest on $\Gamma$
and $\delta=\deg(\mu)$. Then  $\delta\geq 2\cdot\frac{5}{2}=5$ by Theorem~\ref{spanning6},
so $\delta-|S|\geq 2$, whence $\calT(G)\leq 6$ by Proposition~\ref{prop:bound_randomforest}(c).
Similarly, in the setting of Theorem~\ref{thm:Tarski56}(ii) we pick an element $a\in S$ of infinite order
and let $\mu$ be a $G$-invariant random spanning forest on $\Gamma$ satisfying the conclusion
of Theorem~\ref{spanning5}. Again its degree $\delta$ is at least $5$, so $\delta-|S|\geq 2$,
whence $\mathcal T(G)\leq 5$ by Proposition~\ref{prop:bound_randomforest}(g).
\end{proof}

\begin{proof}[Proof of Proposition~\ref{prop:bound_randomforest}]
(a) We first claim that there exists an ordinary
spanning forest $\mathcal F$ on $\Gamma$ (depending on $A$) such that
\begin{equation}
\label{eq:forest}
\sum_{g\in A}\deg_{\mathcal F}(g)\geq \delta|A|.
\end{equation}
Indeed, consider the function $\phi\colon \Sigma_{\Gamma}\to \dbZ_{\geq 0}$
given by $\phi(F)=\sum_{g\in A}\deg_F(g)$.
Integrating $\phi$ with respect to $\mu$, we have
$$\int_{\Sigma_{\Gamma}} \phi\,\, d\mu=\sum_{g\in A}\deg_{\mu}(g)=\delta|A|.$$
Since $\mu$ is a probability
measure supported on forests, we deduce that $\phi(\mathcal F)\geq \delta|A|$ for some spanning forest
$\mathcal F$ on $\Gamma$.
\vskip .12cm

Let $E$ be the set of all directed edges  $(gs,g)$ such that $g\in A$, $s\in S\cup S^{-1}$ and the unoriented edge
$\{gs,g\}$ lies in $\mathcal F$. Let $E_1$ be the subset of $E$ consisting of all edges $(gs,g)\in E$ with $s\in S^{-1}\setminus S$
(the set $E_1$ may be empty, for example, if $S=S^{-1}$).
Note that $|E|\geq \delta |A|$ by \eqref{eq:forest}, and it is clear that $|E_1|\geq |E|-|S||A|$,
so that $|E_1|\geq (\delta-|S|)|A|$.

Since the sets $S^{-1}\setminus S$ and $(S^{-1}\setminus S)^{-1}$ are disjoint, $E_1$ does not contain
a pair of opposite edges. Also note that endpoints of edges in $E_1$ lie in the set $AT^{-1}$.
Let $\Lambda$ be the unoriented graph with vertex set $AT^{-1}$
and edge set $E_1$ (with forgotten orientation). Then $\Lambda$ is a subgraph of $\mathcal F$;
in particular $\Lambda$ is a (finite) forest. Hence
$$|AT^{-1}|=|V(\Lambda)|>|E(\Lambda)|=|E_1|\geq (\delta-|S|)|A|,$$
as desired.

\vskip .12cm
(b) This result is virtually identical to that of \cite[Theorem 4.5]{LPV}, but for completeness
we reproduce the argument. Keeping the notations introduced in the proof of (a), consider
the graph $\Lambda'$ with vertex set $A(T\cup T^{-1})$ and edge  set $E$ with forgotten orientation.
Note that $E$ may contain pairs of opposite edges. However, all such edges must have both their
endpoints in $A$, and since all those edges came from the forest $\mathcal F$, the number of
pairs of opposite edges in $E$ is less than $|A|$. Hence $|E(\Lambda')|\geq |E|-|A|\geq (\delta-1)|A|$,
and we are done by the same argument as in (a).

\vskip .12cm
(c) Since $\delta-|S|\geq 2$, combining (a) with Theorem \ref{HallRado}(i), we deduce that
the (oriented) Cayley graph $Cay(G,T)$ contains a spanning $2$-subgraph.
To finish the proof, by Lemma~\ref{paradox}, it suffices to show that any $2$-subgraph
of $Cay(G,T)$ can be colored to yield an evenly colored $2$-subgraph of the colored Cayley graph
$Cay(G,(T_1,T_2))$. The latter is clear since by definition of $T_1$ and $T_2$, any two-element
subset of $T$ can be ordered so that the first element lies in $T_1$ and the second element lies in $T_2$.
\vskip .12cm
The proof of (d) is analogous to (c).
\vskip .12cm

(e) As in (a), there exists a spanning forest $\mathcal F$ on $\Gamma$ depending on $A$
such that $$\sum_{g\in A}\deg_{\mathcal F}(g)\geq \delta|A|.$$ Define the sets $E$ and $E_1$
as in the proof of (a); again we have $|E|\geq \delta|A|$ and $|E_1|\geq (\delta-|S|)|A|$

Let $E_2$ denote the set of edges $(gs,g)\in E_1$ such that $s\neq a^{-1}$
(equivalently, $E_2$ is the set of edges $(gs,g)$ in $\mathcal F$ with $g\in A$ and
$s\in S^{-1}\setminus (S\cup \{a^{-1}\})$). Then clearly $|E_2|\ge  |E_1|-|A|\ge (\delta-|S|-1)|A|$.

Let $E_3$ be the set of directed edges $(ga^{-1},g)$ for $g\in B$. Clearly, $E_2$ and $E_3$ are disjoint sets and
$E_2\cup E_3$ does not contain a pair of opposite edges.
The endpoints of edges in $E_2\cup E_3$ lie in the set $A T_1^{-1}\cup B T_2^{-1}$,
so as in (a) we have
$$|A T_1^{-1}\cup B T_2^{-1}|> |E_2|+|E_3|\ge (\delta-|S|-1)|A|+|B|.$$
\vskip .12cm
(f) is proved by modifying the proof of (e) in the same way the proof of (b) was obtained from
the proof of (a). Finally, (g) and (h) follow directly from Theorem \ref{HallRado}(ii) and (e) and (f), respectively.
\end{proof}

\begin{Remark}\label{r:9} Informally speaking, the result of Proposition~\ref{prop:bound_randomforest}(a) can only be useful
if the intersection $S\cap S^{-1}$ is small. In particular, if $S$ is symmetric (that is, $S=S^{-1}$),
the proof shows that the set $E_1$ is empty and hence the obtained inequality is vacuous. Note that
Proposition~\ref{prop:bound_randomforest}(b) implies that for any $S$ and any finite set $A$ either
$|A(S\cup \{1\})|\geq  \frac{\delta}{2} |A|$ or $|A (S^{-1}\cup\{1\})|\geq \frac{\delta}{2} |A|$. However, this observation cannot be used to strengthen the result of (a) because we cannot guarantee that one of these inequalities holds for every $A$.
\end{Remark}

\begin{Remark}
One can construct groups with Tarski number $5$ (resp. with Tarski number $6$) and any given minimal number of generators $d\geq 2$. For $d=2$ this follows from Theorem~\ref{thm:Tarski56}(ii) (resp. Theorem \ref{thm:Tarski56}(i))~and~Theorem~\ref{t:hh}(b). For $d\geq 3$ one can take the direct product of a group $G$ from Theorem \ref{thm:Tarski56}(ii) (resp. Theorem \ref{thm:Tarski56}(i)) and a finite elementary Abelian group $C_2^k$ for a suitable $k$.
\end{Remark}

\subsection{Further results and open questions}

We begin this subsection with two open problems:

\begin{Problem} Given $m\in\dbN$, what is the minimal
possible Tarski number of a group from $\Fin_m$ (resp. $\Amen_m$)?
\end{Problem}

Lemma ~\ref{obs:Ozawa} shows that $2m+4$ (resp. $m+3$) is a lower bound for groups from $\Fin_m$ (resp. $\Amen_m$). By Theorem~\ref{thm:Tarski56}(i) this lower bound is exact
for $m=1$, but we do not know exact value already for $m=2$.

\begin{Problem}\label{p:89} Let $G$ be a finitely generated group with $\mathcal C(G)>1$.
\begin{itemize}
\item[(a)] Is it true that $\calT(G)\leq 6$?
\item[(b)] If $G$ is non-torsion, is it true that $\calT(G)\leq 5$?
\item[(c)] If the answer to (a) or (b) is negative, is it at least true that $\calT(G)\leq C$ for
some absolute constant $C$?
\end{itemize}
\end{Problem}

We shall give three pieces of evidence that the answer to at least part (c) of Problem~\ref{p:89} might be positive.

\subsubsection{A result by Peterson and Thom} By \cite[Corollary~4.4]{PT}
any torsion-free group $G$ which has positive first $L^2$-Betti number and
satisfies the Atiyah zero divisor conjecture must contain a non-Abelian free subgroup
and hence has Tarski number $4$. It is feasible that this result remains true if the condition
$\beta_1(G)>0$ is replaced by $\mathcal C(G)>1$ (as we have already mentioned, $\mathcal C(G)\ge \beta_1(G)+1$ for any countable infinite group $G$, and there are no examples where the inequality is known to be strict).

\subsubsection{Tarski half numbers}

Given a paradoxical decomposition of some group with $m$ pieces in one half and $n$ pieces in the other half, we will call the number $\min(m,n)$ the {\it half number} of that decomposition. Given a group $G$, the minimal half number of all its paradoxical decompositions will be called the \emph{Tarski half number} of $G$. Clearly the Tarski half number is at least 2 and can only increase when we pass to subgroups or factor-groups. A natural question is to determine all Tarski half numbers of groups. Since the Tarski half number of a group from $\Fin_n$ cannot be smaller than $n+2$ by Lemma \ref{obs:Ozawa}, there exist groups with arbitrarily large Tarski half numbers.

The next theorem shows that the Tarski half numbers of non-torsion groups with cost $>1$ are minimal possible.

\begin{Theorem}
\label{pd:cost1nontorsion}
Let $G$ be a finitely generated non-torsion group with $\calC(G)>1$. Then the Tarski half number of $G$ is $2$.
\end{Theorem}

\begin{proof} By Proposition~\ref{prop:bound_randomforest}(h), to prove Theorem~\ref{pd:cost1nontorsion},
it suffices to show that there exists a finite subset $S'$ of $G$ containing an element of infinite order $g$
such that the unoriented Cayley graph $\Gamma'=\Cayu(G,S'\cup (S')^{-1})$ admits a $G$-invariant random spanning forest which
has expected degree $>4$ and almost surely contains all edges of $\Gamma'$ labeled by $g^{\pm 1}$. The
construction of such a forest given below is almost identical to the construction from the proof of \cite[Lemma 5]{Th}.

We start with any finite generating set $S$ of $G$ containing an element $a$ of infinite order.
Let $\mu$ be a $G$-invariant random spanning forest on $\Gamma=\Cay_{uo}(G,S\cup S^{-1})$
satisfying the conclusion of Theorem~\ref{spanning5}. In particular, by our assumption, $\deg(\mu)>2$.

Fix $b\in S\setminus\{a,a^{-1}\}$ such that $\mu$ contains the (unoriented) edge $\{e,b\}$ with positive probability $\delta$ (such $b$ exists since $\deg(\mu)>2$), and fix $n$ such that $\delta n\geq 2$.
Let $S'=\{s_0,\ldots, s_n\}\subseteq G$ where $s_i=a^{i} b a^{-i}$ for $0\leq i\leq n-1$ and $s_n=a^n$.
We claim that $\Gamma'=\Cayu(G,S'\cup (S')^{-1})$ admits
a $G$-invariant random spanning forest with the desired properties with $g=a^n$.

To construct $\mu'$, define the map $\theta$ from spanning subgraphs of $\Gamma$
to spanning subgraphs of $\Gamma'$ as follows. Given $g\in G$ and $s'\in S'$, let
$\gamma(g,gs')$ be the ``natural'' path from $g$ to $gs'$ in $\Gamma$ 
(if $s'=a^n$, then $\gamma(g,gs')$ is the path $(\{g,ga\},\ldots, \{ga^{n-1},ga^n\})$
of length $n$, and if $s'=a^i b a^{-i}$, then $\gamma(g,gs')$ is the path
$(\{g,ga\},\ldots, \{ga^iba^{1-i},ga^i ba^{-i}\})$ of length $2i+1$).
Given a spanning subgraph $\Lambda$ of $\Gamma$, we define $\theta(\Lambda)$ to be the
spanning subgraph of $\Gamma'$ whose edge set consists of all edges of the form $\{g,gs'\}$ with $s'\in S'$ for which
the path $\gamma(g,gs')$ is contained in  $\Lambda$. The key observation is that $\theta$ maps
forests to forests. Indeed, assume that $F$ is a forest, but $\theta(F)$ contains a non-trivial reduced
cycle $c$. Choose an oriented cycle in the (oriented) Cayley graph $\Cay(G,S'\cup (S')^{-1})$ which
projects onto $c$, and let $w$ be its label. Then $w$ is a non-trivial reduced group word in the alphabet $S'$.  Since the subset $\{x^n\}\cup\{x^i y x^{-i}: 0\leq i\leq n-1\}$ of the free group $F(x,y)$ freely generates a free subgroup,
the word $\bar w$ obtained from $w$ by substituting elements of $S'$ by the corresponding words in $a,b$ is freely non-trivial.
By definition of $\theta(F)$, the word $\bar w$ must label a non-trivial cycle in $Cay(G,S\cup S^{-1})$
which yields a non-trivial cycle in $F$, which is impossible.

It is easy to see that $\theta$ is a Borel map, so we can define the
probability measure $\mu'$ on spanning subgraphs of $\Gamma'$ by setting $\mu'(B)=\mu(\theta^{-1}(B))$
for any Borel subset $B$ of $\{0,1\}^{\Gamma'}$. Since $\theta$ sends  forests
to forests and $\mu$ is a random spanning forest on $\Gamma$, it follows that $\mu'$
is a random spanning forest on $\Gamma'$.

Since $\mu$ is an invariant spanning forest, $\mu'$ is also invariant. Since $\mu$ contains the (unoriented) edge $\{e,b\}$ with positive probability $\delta$ and the edge $\{e,a\}$ with probability 1, and since $\mu$ is invariant, $\mu$ contains each path $\gamma(e,es_i)$, $0\le i\le n-1$, starting at $e$  with probability $\delta$. Since $\mu'$ also contains each of  the edges $\{e,a^{\pm n}\}$ with probability $1$, the expected degree of $\mu'$ is at least $\delta n+2\geq 4$, as required.
\end{proof}

\begin{Remark} In the case when the group $G$ has infinitely many subgroups of finite index (including the case of infinite residually finite groups) one can give an easier proof of Theorem \ref{pd:cost1nontorsion}. Indeed, if $\calC(G)\ge 2$, the result follows directly from Proposition~\ref{prop:bound_randomforest}(h) and
Theorem~\ref{spanning5}. If $H$ is a finite index subgroup of $G$, then by \cite[Theorem 3]{Gab}, $\calC(H)-1=[G:H](\calC(G)-1)$.
Thus, if $1<\calC(G)< 2$ and $G$ has finite index subgroups of arbitrarily large
index, then some finite index subgroup of $G$ has cost $\geq 2$ yielding the result.
\end{Remark}

\begin{Problem} Is it true that the Tarski half number of any torsion group $G$ with $\calC(G)>1$ is equal to 3?
\end{Problem}

\subsubsection{Cost of our examples of groups with large Tarski numbers}

Below we will show that all groups with Tarski numbers $>6$ obtained using Theorem \ref{Tar_combined} and our proofs of Theorems \ref{t:ej} and \ref{unbounded_2gen} have cost $1$ (and hence are not counterexamples to Problem \ref{p:89}).

Note that each of these groups $G$ has an amenable normal subgroup $N$ such that $G/N$ is an extension of a group from $\Amen_2$ of unbounded exponent by an amenable group. By \cite[Proposition~35.2]{KM}
 if a group $G$ has an infinite normal subgroup $N$ with fixed price and $\mathcal C(N)=1$, then $G$ has fixed price and $\mathcal C(G)=1$. Since every infinite amenable group has fixed price and cost $1$
(see, for example, \cite[Corollary~31.2]{KM}), and an extension of a finite group by a group from $\Amen_2$ is in $\Amen_2$, it is enough to show that the cost of any group from $\Amen_2$ of unbounded exponent is $1$.

\begin{Theorem}
\label{claim:fin2b0}
Let $G$ be a finitely generated group in $\Amen_2$ which does not have bounded exponent.
Then $\calC(G)=1$.
\end{Theorem}

The proof of Theorem~\ref{claim:fin2b0} is based on the following result:

\begin{Theorem} (see \cite[Proposition~32.1(iii)]{KM})
\label{prop321} Let $G$ be a group generated by a countable family
of subgroups $\{G_i\}_{i\in I}$, and let $K=\cap G_i$. Assume that $K$ is amenable
and each $G_i$ has fixed price. Then
$$\calC(G)-\calC(K)\leq \sum_{i\in I} (\calC(G_i)-\calC(K)).$$
\end{Theorem}

\begin{proof}[Proof of Theorem~\ref{claim:fin2b0}]
Let $S=\{s_1,\ldots, s_m\}$ be a finite generating set of $G$.
Fix $N\in\dbN$. By assumption there exists an element $g_N$ of $G$ whose order
is at least $N$. Consider the subgroups $G_1=\la s_1, g_N\ra, \ldots,
G_m=\la s_m, g_N\ra$. Then each $G_i$ is amenable and hence has fixed price
and cost $\leq 1$. The intersection $K=\cap G_i$ is also amenable and has cost
between $1-\frac{1}{N}$ and $1$ (since $|K|\geq N$). Thus, applying Theorem~\ref{prop321}, we deduce
that $\calC(G)\leq 1+\frac{m}{N}$. Therefore letting $N$ tend to $\infty$,
we conclude that $\calC(G)\leq 1$ (and the opposite inequality $\calC(G)\geq 1$
holds since $G$ is infinite).
\end{proof}

\begin{Remark} We expect that the cost
of any group $G$ in $\Amen_2$ is at most $1$. Theorem~\ref{claim:fin2b0} shows that we need to consider only groups from $\Amen_2$ of bounded exponent. A conjecture by Shalom \cite[Section 5.IV]{Shalom} says that every finitely generated group of bounded exponent has property (T). If that was the case (which is hard to believe), every such group would have $\beta_1(G)=0$ \cite[Corollary 6]{BV} and if problem 4.2 from \cite{Gr} has positive solution, then for every such group  $\mathcal C(G)=1$. Note also that by a result of Zelmanov \cite{Zel} for every prime $p$ there exists a number $n=n(p)$ such that every group of exponent $p$ in $\Fin_n$ is finite. It is believable that the minimal such $n(p)$ is $2$, that one can replace $\Fin_n$ by $\Amen_n$, and that the result holds for non-prime numbers $p$. This would also imply that the cost of any group in $\Amen_2$ is at most $1$.
\end{Remark}

\appendix
\section{Equivalence of two definitions of Tarski numbers}

\begin{Theorem}\label{p:1} Let $G$ be a group and $k=\calT(G)$.
Then there exists a paradoxical decomposition of $G$ with pieces $P_1,
\ldots,P_n$, $Q_1,\ldots,Q_m$ and translating elements $g_1,\ldots,g_n,$ $h_1,\ldots,h_m,$ $n+m=k$, as in Definition \ref{d1},
such that the union $\bigcup P_i\cup\bigcup Q_j$ is the whole $G$, the translated sets $P_ig_i$ are disjoint, and the translated sets $Q_jh_j$ are disjoint.
\end{Theorem}

The following argument is very close to a translation of the proof of \cite[Proposition~1.2]{RY}
into a graph-theoretic language.

\begin{proof}
Suppose that $G$ has a paradoxical decomposition with translating sets $S_1$ and $S_2$, with $1\in S_1$ (we can assume that by Remark \ref{1}). By Lemma ~\ref{paradox}, $\Gamma=\Cay(G,(S_1,S_2))$ has a spanning evenly colored $2$-subgraph $\Lambda$.

Let $A$ be the set of vertices which have no outgoing edge in $\Lambda$. For each $g\in A$ consider the
unique oriented path in $\Lambda$ starting from $g$ and going backwards in which all edges have color $1$. All such paths will
clearly be disjoint. Let $\Lambda'$ be the graph obtained from $\Lambda$ by first
removing all the edges from those paths and then adding a loop of color $1$ at all the
vertices on those paths. Then $\Lambda'$ is a spanning evenly colored $2$-subgraph of $\Gamma$
with exactly one outgoing edge at every vertex.

By the same argument as in Lemma~\ref{paradox}, the graph $\Lambda'$ yields a $2$-paradoxical decomposition having the required properties, with the same translating sets, $S_1$ and $S_2$.
\end{proof}

\section{Explicit construction of groups with Tarski number $5$ and groups with Tarski number $6$}

As explained in the introduction, the problem of finding explicit examples of groups with Tarski number $6$
reduces to an explicit construction of $d$-generated torsion groups $G$ with $\beta_1(G)> d-1-\eps$
whose existence is proved in \cite{Os} (to produce groups with Tarski number $6$ we take $d=3$ and $\eps=1/2$).
Such groups are constructed inductively in \cite{Os}, but the proof shows that they are given by presentations of the form
$\la x_1,\ldots,x_d \mid r_1^{n_1},r_2^{n_2},\ldots\ra$ where $r_1,r_2,\ldots$ is a sequence of all elements of
the free group on $x_1,\ldots, x_d$ listed in some order and $n_1,n_2,\ldots$ is some integer sequence. Moreover, given $\eps>0$,
one can specify explicitly how fast the sequence $\{n_i\}$ must grow to ensure that $\beta_1(G)>  d-1-\eps$
for the resulting group $G$.

To get groups with Tarski number $5$, we need a $d$-generated group $G$ without non-Abelian free subgroups, in which at least one of the $d$ generators is of infinite order and $\beta_1(G)> d-1-\eps$ (once again, $d=3$ and $\eps=1/2$ yield groups with Tarski number $5$). While the existence of such groups is not proved explicitly in \cite{Os}, a slight modification of the proof yields the result. Indeed, replace the sequence $r_1,r_2,\ldots$ by a sequence of all the elements of the derived subgroup of the free group on $x_1,\ldots,x_d$. As before, if  $n_1,n_2,\ldots$ is a sufficiently fast growing sequence of natural numbers, the group $\la x_1,\ldots,x_d \mid r_1^{n_1},r_2^{n_2},\ldots\ra$ will have a large  enough first $L^2$-Betti number. Since it is torsion-by-Abelian, it does not contain a free non-Abelian subgroup and clearly the generators $x_1,\ldots,x_d$ have infinite order in $G$.

The goal of this section is to show that a group given by such ``torsion'' or ``torsion-by-Abelian'' presentation has Tarski number $6$ or Tarski number $5$, respectively, under much milder
conditions on the exponents $\{n_i\}$ (see Theorem~\ref{PTP} below). Note that we will not be able to control the
first $L^2$-Betti number of such a group $G$, but we will estimate the first $L^2$-Betti number of
some quotient $Q$ of $G$, which is sufficient for producing groups with Tarski number $6$. To get groups with Tarski number
$5$, we will also ensure that the image of (at least one of) the generators of $G$ inside $Q$ has infinite order. Note that since we do not know the exact value of the Tarski number of the free Burnside group of a sufficiently large odd exponent (we only know by Theorem \ref{t:csgh}(ii) that it is between $6$ and $14$), it is possible that one can have a constant sequence $n_1,n_2,\ldots$, say, $n_i=665$, $i\in\mathbb{N}$, and still get a group with Tarski number $6$.

Before proceeding, we introduce some additional terminology. Given a group $G$, denote by $G^{\abtf}$ the largest
torsion-free abelian quotient of $G$ which will be referred to as the {\it torsion-free abelianization} of $G$
(such quotient exists by \cite[Theorem~5, p.231]{Ma}). Also note that $G^{\abtf}$ can be defined as the quotient of the usual abelianization $G/[G,G]$ by its torsion subgroup.
It is easy to see that the correspondence $G\mapsto G^{\abtf}$ is functorial (see \cite[Theorem~1, p.229]{Ma}).
The kernel of the canonical map $G\to G^{\abtf}$ will be denoted by $[G,G]^{\rm tf}$.

\begin{Lemma}
\label{lem:abtf} Let $G$ be a finitely generated group.
\begin{itemize}
\item[(a)] Let $p$ be a prime, and let $G_p$ be the image of $G$ in its pro-$p$ completion.
Then the torsion-free abelianizations of $G$ and $G_p$ are (naturally) isomorphic.
\item[(b)] Let $N_0=\{1\}\subseteq N_1 \subseteq \ldots$ be an ascending chain of normal subgroups of $G$,
let $G(i)=G/N_i$, and let $G(\infty)=G/\cup_{i=1}^{\infty} N_i$. Then for all sufficiently large $n$,
the induced map $G(n)^{\abtf}\to G(\infty)^{\abtf}$ is an isomorphism.
\end{itemize}
\end{Lemma}
\begin{proof} Part (b) is obvious. To prove (a), denote by $K$ the kernel of the canonical map
$G\to G_p$. Notice that $K$ is contained in the kernel of  any homomorphism from $G$ to a residually-$p$ group.
The group $G/[G,G]^{\rm tf}$ is finitely generated free abelian and so residually-$p$, whence $K\subseteq [G,G]^{\rm tf}$.
Therefore, the induced map $G^{\abtf}\to (G_p)^{\abtf}=(G/K)^{\abtf}$ is an isomorphism.
\end{proof}

The proof of Theorem~\ref{PTP} below mostly utilizes ideas from \cite{Os} and \cite{LuOs} where similar results were proved.

\begin{Theorem}
\label{PTP}
Let $X$ be a finite set, $F(X)$ the free group on $X$,
$p$ a prime, $r_1,r_2,\ldots$ a finite or infinite sequence of elements of $F(X)$, and $R=\{r_i^{p^{n_i}}, i=1, 2,...\}$ for some integer sequence $n_1,n_2,\ldots$.
Let $G=\la X| R\ra$. Then $G$ has a quotient $Q$ such that $G$ and $Q$ have the same torsion-free abelianization and
\begin{equation}
\label{eq:beta}
\beta_1(Q)\geq |X|-1-\sum_{i}\frac{1}{p^{n_i}}.
\end{equation}
In particular,
\begin{enumerate}
\item  If $|X|=3$, $\sum_i\frac{1}{p^{n_i}}\leq \frac{1}{2}$ and
the sequence $\{r_i\}$ ranges over the whole free group $F(X)$, then
$G$ and $Q$ are torsion and $\calT(G)=\calT(Q)=6$.
\item If $X=\{x_1,x_2,x_3\}$, $\sum_i\frac{1}{p^{n_i}}\leq \frac{1}{2}$ and
the sequence $\{r_i\}$ ranges over the whole derived subgroup $\gamma_2F(X)$,  then
$G$ and $Q$ are torsion-by-Abelian, the image of $x_1$ in $Q$ has infinite order and $\calT(G)=\calT(Q)=5$.
\end{enumerate}
\end{Theorem}

We start by stating (a special case of) a result of Peterson and Thom~\cite[Theorem~3.2]{PT} which
is similar to Theorem~\ref{PTP}:
\begin{Theorem}[\cite{PT}]
\label{thm:PT}
Let $G$ be a group given by a finite presentation
$\la X \mid r_1^{m_1},\ldots, r_k^{m_k}\ra$ for some $r_1,\ldots, r_k\in
F(X)$ and $m_i\in\dbN$. Assume that for each $1\leq i\leq k$, the order of $r_i$ in $G$ is equal to $m_i$. Then $\beta_1(G)\geq |X|-1-\sum_{i=1}^k \frac{1}{m_i}$.
\end{Theorem}
In general, the assumption on the orders of $r_i$ cannot be eliminated since, for instance,
the trivial group has a presentation $\la x,y \mid x^m, x^{m+1}, y^m,y^{m+1}\ra$ for any $m\in\dbN$. If all $m_i$ are powers of a fixed prime $p$, it is possible that
Theorem~\ref{thm:PT} holds without any additional restrictions, but we are not able to
prove such a statement. What we can prove is the following variation:
\begin{Proposition}
\label{PT_finite}
Let $G$ be a group given by a finite presentation
$\la X \mid r_1^{m_1},\ldots, r_k^{m_k}\ra$ for some $r_1,\ldots, r_k\in
F(X)$, where each $m_i$ is a power of some fixed prime $p$. Let $G_{p}$ be the image of $G$ in its pro-$p$ completion
$\widehat G_p$. Then
$\beta_1(G_p)\geq |X|-1-\sum_{i=1}^k \frac{1}{m_i}$.
\end{Proposition}
Before establishing Proposition~\ref{PT_finite} we show how Theorem~\ref{PTP} follows from it.

\begin{proof}[Proof of Theorem~\ref{PTP}] If the sequence $\{r_i\}$ is finite,
then the group $Q=G_p$ satisfies \eqref{eq:beta} by Proposition~\ref{PT_finite}
and has the same torsion-free abelianization as $G$ by Lemma~\ref{lem:abtf}(a).

If $\{r_i\}$  is infinite, let $R_m=\{r_i^{p^{n_i}}\}_{i=1}^m$ and $G(m)=\la X\mid R_m\ra$.
Let $\beta=|X|-1-\sum\limits_{i=1}^{\infty}\frac{1}{p^{n_i}}$.
Then $\beta_1(G(m)_p)\geq \beta$
for each $m$ by Proposition~\ref{PT_finite}. Note that $G(m+1)_p$ is a quotient of $G(m)_p$. Let $Q=\varinjlim G(m)_p$, that is, if $G(m)_p=F(X)/N_m$, put
$Q=F(X)/\cup_{m\in\dbN}N_m$. Then $Q$ is clearly a quotient of $G$; on
the other hand, the sequence $\{G(m)_p\}$ converges to $Q$ in the space of marked
groups, and therefore by a theorem of Pichot~\cite[Theorem~1.1]{Pi} we have $\beta_1(Q)\geq \limsup\limits_m \beta_1(G(m)_p)\geq \beta$.

Finally, for any sufficiently large $m$ we
have $G^{\abtf}\cong G(m)^{\abtf}\cong (G(m)_p)^{\abtf}\cong Q^{\abtf}.$
by Lemma~\ref{lem:abtf}.
\end{proof}

We proceed with the proof of Proposition~\ref{PT_finite}.
Below $p$ will be a fixed prime.
Let $F$ be a free group. Given an element $f\in F$, define $s(f)\in F$ and $e(f)\in \dbN$
by the condition that $f=s(f)^{p^{e(f)}}$ and $s(f)$ is not a $p^{\rm th}$-power in $F$.
The following definition was introduced by Schlage-Puchta in \cite{SP}:

\begin{Definition}\label{p-def}\rm Given a presentation $(X,R)$ by generators and relations with $X$ finite,
define its $p$-deficiency $\defi_p(X,R)$ by $\defi_p(X,R)=|X|-1-\sum_{r\in R}\frac{1}{p^{e(r)}}$.
\end{Definition}
In terms of the $p$-deficiency, Proposition~\ref{PT_finite} reduces to the following result:

\begin{Proposition}\label{B.5}
\label{PT_finite2}
Let $(X,R)$ be a finite presentation of a group $G$.
Then $\beta_1(G_p)\geq def_p(X,R)$.
\end{Proposition}

As usual, for a finitely presented group $G$ we define $\defi(G)$ to be the maximal
possible value of the difference $|X|-|R|$ where $(X,R)$ ranges over all finite presentations
of $G$.

\begin{Definition}[\cite{LuOs}]\rm Given a finitely presented group $G$, define
the quantity $\mathit{vdef}_p(G)$ by $\mathit{vdef}_p(G)=\sup_H \frac{\defi(H)-1}{[G:H]}$
where $H$ ranges over all normal subgroups of $G$ of $p$-power index.
\end{Definition}\rm

\begin{Definition}[\cite{ErLu}]\rm A presentation $(X,R)$ will be called {\it $p$-regular}
if for any $r\in R$ the element $s(r)$ has order (precisely) $p^{e(r)}$ in the group $\la X|R \ra_p$.
\end{Definition}

According to \cite[Lemma~3.6]{LuOs}, for any finitely presented group $G$
we have $\beta_1(G_p)\geq vdef_p(G)$. On the other hand, by \cite[Lemma~5.5]{ErLu},
if a group $G$ has a finite $p$-regular presentation $(X,R)$,
then $vdef_p(G)\geq def_p(X,R)$. These two results imply Proposition~\ref{B.5}
in the case of $p$-regular presentations. The proof in the general case will be completed via the following lemma.

\begin{Lemma}
\label{lem:pregular}
Let $(X,R)$ be a finite presentation. Then there exists a subset $R'$
of $R$ such that the presentation $(X,R')$ is $p$-regular and the natural surjection
$\la X|R'\ra \to \la X|R\ra$ induces an isomorphism of pro-$p$ completions
$\widehat{\la X|R'\ra}_p\to \widehat{\la X|R\ra}_p$ and hence also an
isomorphism of $\la X|R'\ra_p$ onto $\la X|R\ra_p$.
\end{Lemma}

\begin{proof}[Proof of Lemma~\ref{lem:pregular}]
Let $G=\la X|R\ra$, and assume that $(X,R)$ is not $p$-regular. Thus there exists
$r\in R$ such that the order of $s(r)$ in $G_p$ is strictly smaller than
$p^{e(r)}$.  We will show that  if we set $R'=R\setminus\{r\}$ and $G'=\la X|R'\ra$,
then the natural map $\widehat{G'}_p\to \widehat{G}_p$ is an isomorphism.
Lemma~\ref{lem:pregular} will follow by multiple applications of this step.

If a discrete group is given by a presentation by generators and relators, its
pro-$p$ completion is given by the same presentation in the category of pro-$p$ groups.
It follows that
\begin{equation}
\label{prop_pres}
 \widehat{G}_p\cong  \widehat{G'}_p/\lla s(r)^{p^{e(r)}}\rra
\end{equation}
where $\lla S \rra$ is the closed normal subgroup generated by a set $S$. Thus,
it is sufficient to show that $s(r)^{p^{e(r)}}= 1$ in $\widehat{G'}_p$.
We will show that already $s(r)^{p^{e(r)-1}}= 1$ in $\widehat{G'}_p$

Let $m$ be the order of $s(r)$ in $\widehat G_p$. Then by
assumption $m<p^{e(r)}$; on the other hand, $m$ must be a power of $p$
(since $\widehat G_p$ is pro-$p$), so $m$ divides $p^{e(r)-1}$.
Thus, if we let $g$ be the image of $s(r)^{p^{e(r)-1}}$ in $\widehat{G'}_p$,
then $g$ lies in the kernel of the homomorphism $\widehat{G'}_p\to \widehat{G}_p$,
whence by \eqref{prop_pres}, $g$ lies in the closed normal subgroup generated by $g^p$.
It is easy to see that this cannot happen in a pro-$p$ group unless $g=1$.
\end{proof}

\section{Golod-Shafarevich groups}

In this section we introduce Golod-Shafarevich groups and give a self-contained proof
of Proposition~\ref{cor:GS1}.

The definitions of Golod-Shafarevich groups and the related notion of
a weight function will be given in a simplified form below since this will be
sufficient for the purposes of this paper.  For more details the reader
is referred to \cite{Er2}.

Let $p$ be a fixed prime number. Given a finitely generated group $G$,
let $\{\omega_n G\}_{n\in\dbN}$ be the {\it Zassenhaus $p$-filtration} of $G$
defined by $\omega_n G=\prod_{i\cdot p^j\geq n}(\gamma_i G)^{p^j}$.
It is easy to see that $\{\omega_n G\}$ is a descending chain of
normal subgroups of $p$-power index in $G$ satisfying
\begin{equation}[\omega_n G, \omega_m G]\subseteq \omega_{n+m}G\quad \mbox{and }\quad
(\omega_n G)^p\subseteq \omega_{np}G.\label{eq:90}\end{equation}
Moreover, $\{\omega_n G\}$ is a base for the pro-$p$ topology on $G$, so in particular, $\cap \omega_n G=\{1\}$ if and only if $G$ is a residually-$p$ group.

Now let $F$ be a finitely generated free group.
Then $F$ is residually-$p$ for any $p$, so for any $f\in F\setminus \{1\}$
there exists (unique) $n\in\dbN$ such that $f\in \omega_n F\setminus \omega_{n+1} F$.
This $n$ will be called the degree of $f$ and denoted $\deg(f)$. We set
$\deg(1)=\infty$

\begin{Definition}\rm Let $F$ be a finitely generated free group.
\begin{itemize}
\item[(i)] A function $W\colon F\to \dbN\cup\{\infty\}$ will be called a {\it weight function}
if $W(f)=\tau^{\deg(f)}$ where $\tau\in (0,1)$ is a fixed real number.
\item[(ii)] If $W$ is a weight function on $F$ and $\pi\colon F\to G$ an epimorphism, then $W$ induces a function on $G$ (also denoted by $W$)
given by
$$W(g)=\inf\{W(f): \pi(f)=g\}$$
Such $W$ will be called a {\it valuation} on $G$.
\item[(iii)] If $W$ is a valuation on $G$, for any countable subset $S$ of $G$
we put $W(S)=\sum_{s\in S}W(s)\in \dbR_{\geq 0}\cup\{\infty\}$.
\end{itemize}
\end{Definition}

The following remark is a reformulation of property (\ref{eq:90}) above.
\begin{Remark}
\label{val:props}
Let $W$ be a valuation on a (finitely generated) group $G$. Then
 for any $g,h\in G$ we have
\begin{itemize}
\item[(i)] $W(gh)\leq\max\{W(g),W(h)\}$ and $W(g^{-1})=W(g)$
\item[(ii)] $W([g,h])\leq W(g)W(h)$
\item[(iii)] $W(g^p)\leq W(g)^p$.
\end{itemize}
\end{Remark}

\begin{Definition}\rm $\empty$
\begin{itemize}
\item[(i)] Let $\la X|R \ra$ be a presentation of a group $G$ with $|X|<\infty$ and $W$ a weight function on $F(X)$. Then we will call the triple $(X,R,W)$ a {\it weighted presentation} of $G$.
\item[(ii)] A weighted presentation $(X,R,W)$ will be called {\it Golod-Shafarevich}
if $$W(X)-W(R)-1>0.$$
\item[(iii)] A finitely generated group $G$ is called {\it Golod-Shafarevich}
(with respect to $p$) if it has a Golod-Shafarevich weighted presentation.
\end{itemize}
\end{Definition}

As was already proved in 1960's, Golod-Shafarevich groups are always infinite; in fact, they have infinite pro-$p$ completions
(see \cite[\S~2-4]{Er2}\footnote{In the foundational paper \cite{GS} the same statement was proved
for a different, although very similar, class of groups.}).
Also by the nature of their definition, any Golod-Shafarevich group has a lot
of quotients which are still Golod-Shafarevich, thanks to the following
observation:

\begin{Remark}
\label{obs:GS}
Let $(X,R,W)$ be a Golod-Shafarevich weighted presentation of a group $G$,
and let $\eps=W(X)-W(R)-1$ (so that $\eps>0$ by assumption).
Then for any $T\subseteq G$ with $W(T)<\eps$, the group $G/\lla T\rra$
is also Golod-Shafarevich (and therefore infinite).
\end{Remark}

The following proposition is a natural generalization of \cite[Theorem~3.3]{Er2}.
In fact, it is a special case of a result from \cite{EJ2} (see \cite[Lemma~5.2]{EJ2} and a remark
after it), but since the setting in \cite{EJ2} is much more general than ours, we present
the proof for the convenience of the reader.

\begin{Proposition}
\label{GS1}
Let $G$ be a group with weighted presentation $(X,R,W)$. Let $\Sigma$ be a countable
collection of finite subsets of $G$ such that $W(S)<1$ for each $S\in\Sigma$.
Then for every $\eps>0$ there is a subset $R_{\eps}$ of $G$
with $W(R_{\eps})<\eps$ and the following property: if $G'=G/\lla R_{\eps}\rra$,
then for each $S\in\Sigma$, the image of $S$ in $G'$ generates a finite group.

In particular, by Remark~\ref{obs:GS},
if the weighted presentation $(X,R,W)$ is Golod-Shafare\-vich,
by choosing small enough $\eps$, we can ensure that $G'$ is Golod-Shafare\-vich.
\end{Proposition}
\begin{proof}
Let $g_1, g_2,\ldots $ be an enumeration of elements of $G$,
and choose integers $n_1,n_2,\ldots$ such that
$\sum_{i\in\dbN} W(g_i^{p^{n_i}})<\eps/2$ -- this is possible
by Remark~\ref{val:props}(iii).

Let $S_1,S_2,\ldots$ be an enumeration of $\Sigma$.
Given $n,k\in\dbN$, let $S_n^{(k)}$ be the set of all left-normed commutators
of length $k$ in elements of $S_n$. Using
Remark~\ref{val:props}(ii) we have
\begin{multline*}
W(S_n^{(k)})=\sum_{h_1,\ldots, h_k\in S_n}W([h_1,\ldots,h_k])\leq \sum_{h_1,\ldots, h_k\in S_n}W(h_1)\ldots W(h_k)=W(S_n)^k,
\end{multline*}
so by our assumption
$W(S_n^{(k)})\to 0$ as $k\to\infty$. Therefore, we can
find an integer sequence $k_1,k_2,\ldots$ such that
$\sum_{n\in\dbN} W(S_n^{(k_n)})<\eps/2$.

Now define $G'=G/\lla R_{\eps}\rra$ where
$R_{\eps}=\{g_i^{p^{n_i}}\}_{i\in\dbN}\cup \bigcup_{n=1}^{\infty} S_n^{(k_n)}.$
Then by construction $W(R_{\eps})<\eps$.
Also by construction, for each $n$ the subgroup generated by the image of $S_n$
in $G'$ is torsion and nilpotent, hence finite.
\end{proof}

We are finally ready to prove Proposition~\ref{cor:GS1} restated below.

\begin{Proposition}
Let $G$ be a Golod-Shafarevich group. Then there exists
a quotient $H$ of $G$ which is also Golod-Shafarevich and satisfies the following
property: for every $n\in\dbN$ there is a finite index subgroup $H_n$ of $H$
such that all $n$-generated subgroups of $H_n$ are finite.
\end{Proposition}

\begin{proof} Let $(X,R,W)$ be a Golod-Shafarevich weighted presentation of $G$.
For every $n\in\dbN$ let $G_n=\{g\in G: W(g)<\frac{1}{n}\}$. Then $G_n$ is a finite
index subgroup of $G$ (more specifically, if $\tau<1$ is such that $W(f)=\tau^{\deg(f)}$ for every $f\in F(X)$,
then $G_n\supseteq \omega_m G$ whenever $\tau^m<\frac{1}{n}$).

Let $\Sigma$ be the collection of all $n$-element subsets of $G_n$, where $n$ ranges over $\dbN$.
By construction $W(S)<1$ for each $S\in\Sigma$, and applying Proposition \ref{GS1} to this collection of subsets,
we obtain a group $H$ with desired properties (where $H_n$ is the image of $G_n$ in $H$).
\end{proof}

\begin{Remark}
We finish with a remark about Theorem~\ref{unbounded_2gen}.
Our original construction of infinite $2$-generated groups with property $(T)$ and unbounded Tarski numbers
was explicit apart from the description of examples of infinite property $(T)$ groups
in $\Fin_m$. Such groups can also be defined by explicit presentations as explained below.

Given an integer $d\geq 2$ and a prime $p$, let $G_{p,d}$ be the group with presentation $\la X|R\ra$
where $X=\{x_1,\ldots,x_d\}$ and $R=\{x_i^p, [x_i,x_j,x_j]\}_{1\leq i\neq j\leq d}$.
By \cite[Theorem~12.1]{Er2}, $G_{p,d}$ is a Golod-Shafarevich group with property $(T)$
whenever $d\geq 9$ and $p>(d-1)^2$. Applying the proof of Proposition~\ref{GS1} to the group $G=G_{p,d}$
and suitable $\Sigma$ and $\eps$, one obtains a concrete example of an infinite group with property $(T)$
which lies in $\Fin_{m}$ for $m<d/2$.

Moreover, observe that the group $G=G_{p,d}$  admits an automorphism $\sigma$
of order $d$ which cyclically permutes the generators. One can show
that the set of relators $R_{\eps}$ in the proof of Proposition~\ref{GS1}
can be chosen $\sigma$-invariant, so that $\sigma$ induces an automorphism $\sigma'$
of the quotient $G'=G/\lla R_{\eps}\rra$. Then the group $G'\rtimes \la \sigma' \ra$
is an infinite $2$-generated group with property $(T)$ whose Tarski number can be made
arbitrarily large by choosing a large enough $d$ (by Theorem~\ref{Tar_combined}(b)).
This provides an alternative proof of Theorem~\ref{unbounded_2gen}.
\end{Remark}


\begin{thebibliography}{AAA}

\bibitem{CSGH}
T.~Ceccherini-Silberstein, R.~Grigorchuk and P.~de~la~Harpe,
\textit{Amenability and paradoxical decompositions for pseudogroups and discrete metric spaces}, Proc. Steklov Inst. Math. 224 (1999), no. 1, 57 --97.

\bibitem{BH} B. Bekka, P. de la Harpe and A. Valette, \textit{Kazhdan's property $(T)$.} New Mathematical Monographs, 11. Cambridge University Press, Cambridge, 2008.

\bibitem{BV} B. Bekka and A. Valette. Group cohomology, harmonic functions and
the first $L^2$-Betti number. Potential Anal., 6(4):313--326, 1997.

\bibitem{EM} I.~Epstein and N.~Monod,
\it Non-unitarisable representations and random forests,
\rm Int. Math. Res. Not. IMRN  2009,  no. 22, 4336–4353.


\bibitem{EJ}
M.~Ershov, \textit{Kazhdan quotients of Golod-Shafarevich groups},
with appendices by A.~Jaikin-Zapirain,
Proc. Lond. Math. Soc. (3)  102 (2011),  no. 4, 599--636.

\bibitem{EJ2}
M.~Ershov and A.~Jaikin-Zapirain,
\it Groups of positive weighted deficiency,
\rm J. Reine Angew. Math.  677 (2013), no. 677, 71--134.


\bibitem{Er2}
M. Ershov,
\it Golod-Shafarevich groups: a survey,
\rm Int. J. Alg. Comp., 22 (2012), no. 5, 68 pages.

\bibitem{ErLu} M.~Ershov and W. L\"uck,
\it The first $L^2$-Betti number and approximation in arbitrary characteristic,
\rm  Doc. Math.  19  (2014), 313–-332.


\bibitem{Gab} D. Gaboriau, Co\^{u}t des relations d'\'equivalence et des groupes. Invent. Math. 139 (2000),
41--98.

\bibitem{Gab1} D. Gaboriau, Invariants $L^2$ de relations d'\'equivalence et de groupes, Publ. math. Inst. Hautes \'Etudes Sci.,  95 (2002), no. 1, 93--150.


\bibitem{Gi} G. Golan,
\it Tarski numbers of group actions,
\rm preprint (2014), \rm http://arxiv.org/abs/1406.5689

\bibitem{Go}
 E.~S.~Golod,
\textit{Some problems of Burnside type},
1968,  Proc. Internat. Congr. Math. (Moscow, 1966)  pp. 284--289, Izdat. "Mir'', Moscow.

\bibitem{GS}
E. Golod and I. Shafarevich,
\it On the class field tower,
\rm Izv. Akad. Nauk SSSR Ser. Mat.  28 (1964), 261--272.

\bibitem{GH} R. Grigorchuk, P. de la Harpe,
\it Limit behaviour of exponential growth rates for finitely generated groups,
\rm Essays on geometry and related topics, Vol. 1, 2, 351--370,
Monogr. Enseign. Math., 38, Enseignement Math., Geneva, 2001.

\bibitem{Gr} {\L}. Grabowski,
\it Open problem session - Oberwolfach, 5-11.09.2010,
\rm \url{http://homepages.warwick.ac.uk/~masmbh/files/oberwolfach_op_session.09.2010.pdf}.


\bibitem{KM} A. Kechris and B. Miller,
\it Topics in orbit equivalence
\rm Lecture Notes in Mathematics, 1852, Springer-Verlag, Berlin (2004).

\bibitem{Luck} W. L\"uck, $L^2$-invariants of regular  coverings of compact manifolds and CW-complexes, Handbook of geometric topology, 735--817, North-Holland, Amsterdam, 2002.

\bibitem{Luck1} W. L\"uck,  $L^2$-invariants: theory and applications to geometry and K-theory.
Ergebnisse der Mathematik und ihrer Grenzgebiete. 3. Folge. A Series of Modern Surveys in Mathematics, 44. Springer-Verlag, Berlin, 2002.

\bibitem{LuOs} W.~L{\"u}ck and D.~Osin.
{\em Approximating the first {$L^2$}-{B}etti number of residually finite
  groups,}
J. Topol. Anal., 3 (2011), no. 2, 153--160.

\bibitem{L} R.~Lyons,
\it Random complexes and $l^2$-Betti numbers.
\rm J. Topol. Anal.  1  (2009),  no. 2, 153--175.

\bibitem{LPV} R. Lyons, M. Pichot and S. Vassout,
\it Uniform non-amenability, cost, and the first $l^2$-Betti number.
\rm Groups Geom. Dyn.  2  (2008),  no. 4, 595--617.

\bibitem{Ma} A. I. Mal'cev,
\it Algebraic systems,
\rm Translated from the Russian by B. D. Seckler and A. P. Doohovskoy, Springer-Verlag, Berlin, Heidelberg and New York (1973).

\bibitem{Ord} B. H. Neumann,
\it On ordered groups.
\rm Amer. J. Math. 71, (1949), 1--18.


\bibitem{HH} B. H. Neumann and H. Neumann,
\it Embedding theorems for groups,
\rm J. London Math. Soc. 34 (1959), 465--479.




\bibitem{Os} D.~Osin,
\textit{$L^2$-Betti numbers and non-unitarizable groups without free subgroups},
Int. Math. Res. Not. 22 (2009),  4220--4231.


\bibitem{OS}
 N.~Ozawa and M.~Sapir,
\textit{Non-amenable groups with arbitrarily large Tarski number?},
mathoverflow question 137678.


\bibitem{PT} J.~Peterson and A.~Thom,
\it Group cocycles and the ring of affiliated operators.
\rm Invent. Math.  185  (2011),  no. 3, 561--592.

\bibitem{Pi}  M.~Pichot,
\it Semi-continuity of the first $l^2$-Betti number on the space of finitely generated groups.
\rm Comment. Math. Helv.  81  (2006),  no. 3, 643--652.




\bibitem{RY} A.~Rejali and A.~Yousofzadeh,
\it Configuration of groups and paradoxical decompositions.
\rm Bull. Belg. Math. Soc. Simon Stevin  18  (2011),  no. 1, 157�-172.

\bibitem{Sa} M.~Sapir,
\emph{Combinatorial algebra: syntax and semantics},\\ http://www.math.vanderbilt.edu/$\sim$msapir/book/b2.pdf.

\bibitem{SP} J.-C. Schlage-Puchta,
\it A $p$-group with positive rank gradient.
\rm J. Group Th. 15 (2012), no. 2, 261--270.


\bibitem{Shalom}  Y. Shalom,
\it The algebraization of Kazhdan's property (T).
\rm International Congress of Mathematicians. Vol. II, 1283--1310, Eur. Math. Soc., Z\"urich, 2006.

\bibitem{Th} A.~Thom,
\it The expected degree of minimal spanning forests,
\rm preprint (2013), arXiv: 1306.0303.

\bibitem{Wa} S.~Wagon,
\it The Banach-Tarski paradox,
\rm Cambridge University Press, (1985).

\bibitem{Zel} E. I. Zelmanov,
\it On additional laws in the Burnside problem on periodic groups.
\rm Internat. J. Algebra Comput. 3 (1993), no. 4, 583--600.
\end{thebibliography}
\end{document}